\documentclass[11pt,reqno,sumlimits]{amsart}
\usepackage{amsfonts, amsmath, amscd, amssymb, euscript, amsthm, array, booktabs, color, dcolumn, shortvrb, tabularx, units, url, mathrsfs, enumitem, tikz-cd, nameref}
\PassOptionsToPackage{bookmarks=false}{hyperref}
\PassOptionsToPackage{linktocpage}{hyperref}
\usepackage[pdfstartview=FitH]{hyperref}
\usepackage[all]{xy}
\usepackage{kpfonts}
\setlist[enumerate,1]{label={(\alph*)}}
\normalsize
\makeatletter
\newcommand{\addresseshere}{\enddoc@text\let\enddoc@text\relax}
\makeatother
\DeclareMathOperator{\op}{op}

\DeclareMathOperator{\Aut}{Aut}
\DeclareMathOperator{\LL}{L}
\DeclareMathOperator{\todd}{Todd}
\DeclareMathOperator{\FL}{FL}
\DeclareMathOperator{\ind}{ind}
\DeclareMathOperator{\vol}{vol}
\DeclareMathOperator{\CS}{CS}
\DeclareMathOperator{\odd}{odd}
\DeclareMathOperator{\even}{even}
\DeclareMathOperator{\im}{Im}
\DeclareMathOperator{\id}{id}
\DeclareMathOperator{\ch}{ch}
\DeclareMathOperator{\End}{End}
\DeclareMathOperator{\rk}{rank}
\DeclareMathOperator{\str}{str}
\DeclareMathOperator{\tr}{tr}
\begin{document}
\setlength{\baselineskip}{1.5\baselineskip}
\theoremstyle{definition}
\newtheorem{coro}{Corollary}
\newtheorem{thm}{Theorem}
\newtheorem{defi}{Definition}
\newtheorem{lemma}{Lemma}
\newtheorem{exam}{Example}
\newtheorem{prop}{Proposition}
\newtheorem{remark}{Remark}
\newcommand{\hto}{\hookrightarrow}
\newcommand{\wt}[1]{{\widetilde{#1}}}
\newcommand{\ov}[1]{{\overline{#1}}}
\newcommand{\un}[1]{{\underline{#1}}}
\newcommand{\wh}[1]{{\widehat{#1}}}
\newcommand{\deff}[1]{{\bf\emph{#1}}}
\newcommand{\abs}[1]{\lvert#1\rvert}
\newcommand{\norm}[1]{\lVert#1\rVert}
\newcommand{\inner}[1]{\langle#1\rangle}
\newcommand{\poisson}[1]{\{#1\}}
\newcommand{\biginner}[1]{\Big\langle#1\Big\rangle}
\newcommand{\set}[1]{\{#1\}}
\newcommand{\Bigset}[1]{\Big\{#1\Big\}}
\newcommand{\BBigset}[1]{\bigg\{#1\bigg\}}
\newcommand{\dis}[1]{$\displaystyle#1$}
\newcommand{\R}{\mathbb{R}}
\newcommand{\EE}{\mathbb{E}}
\newcommand{\GG}{\mathbb{G}}
\newcommand{\N}{\mathbb{N}}
\newcommand{\Z}{\mathbb{Z}}
\newcommand{\Q}{\mathbb{Q}}
\newcommand{\E}{\mathcal{E}}
\newcommand{\T}{\mathcal{T}}
\newcommand{\G}{\mathcal{G}}
\newcommand{\F}{\mathcal{F}}
\newcommand{\I}{\mathcal{I}}
\newcommand{\V}{\mathcal{V}}
\newcommand{\W}{\mathcal{W}}
\newcommand{\SSS}{\mathcal{S}}
\newcommand{\h}{\mathbb{H}}
\newcommand{\C}{\mathbb{C}}
\newcommand{\A}{\mathcal{A}}
\newcommand{\LLL}{\mathcal{L}}
\newcommand{\HH}{\mathcal{H}}
\newcommand{\PP}{\mathcal{P}}
\newcommand{\K}{\mathcal{K}}
\newcommand{\RRR}{\mathscr{R}}
\newcommand{\AAA}{\mathscr{A}}
\newcommand{\DDD}{\mathscr{D}}
\newcommand{\e}{\mathscr{E}}
\newcommand{\hh}{\mathscr{H}}
\newcommand{\w}{\mathscr{W}}
\newcommand{\f}{\mathscr{F}}
\newcommand{\z}{\mathcal{Z}}
\newcommand{\g}{\mathscr{G}}
\newcommand{\so}{\mathfrak{so}}
\newcommand{\gl}{\mathfrak{gl}}
\newcommand{\aaa}{\mathbb{A}}
\newcommand{\bbb}{\mathbb{B}}
\newcommand{\ttt}{\mathbb{T}}
\newcommand{\DD}{\mathsf{D}}
\newcommand{\ff}{\mathsf{F}}
\newcommand{\FF}{\mathbb{F}}
\newcommand{\ccc}{\bold{c}}
\newcommand{\sss}{\mathbb{S}}
\newcommand{\cdd}[1]{\[\begin{CD}#1\end{CD}\]}
\numberwithin{equation}{subsection}
\numberwithin{thm}{section}
\numberwithin{lemma}{section}
\numberwithin{coro}{section}
\numberwithin{prop}{section}
\numberwithin{remark}{section}
\numberwithin{defi}{section}
\normalsize
\title[An extended variational formula for the Bismut--Cheeger eta form]{An extended variational formula for the Bismut--Cheeger eta form and its applications}
\author{Man-Ho Ho}
\email{homanho@bu.edu}
\address{Hong Kong}
\maketitle
\nocite{*}
\begin{abstract}
The purpose of this paper is to extend our previous work on the variational formula for the Bismut--Cheeger eta form without the kernel bundle assumption by allowing the spin$^c$ Dirac operators to be twisted by isomorphic vector bundles and to establish the $\Z_2$-graded additivity of the Bismut--Cheeger eta form. Using these results, we give alternative proofs of the fact that the analytic index in differential $K$-theory is a well defined group homomorphism and the Riemann--Roch--Grothendieck theorem in $\R/\Z$ $K$-theory.
\end{abstract}
\tableofcontents

\section{Introduction}

The Bismut--Cheeger eta form serves as a transgression form between the Chern character of the index bundle and its Atiyah--Singer representative at the differential form level \cite[Theorems 4.35 and 4.95]{BC89} and \cite[Theorem 0.1]{D91}. In this paper, we extend our previous work on the variational formula for the Bismut--Cheeger eta form \cite[Proposition 1]{H20} by allowing the spin$^c$ Dirac operators to be twisted by isomorphic vector bundles. In addition, we prove the $\Z_2$-graded additivity of the Bismut--Cheeger eta form. We then present some applications of these results. All the Dirac operators in this paper are not assumed to satisfy the kernel bundle assumption.

To put the paper into context, let $\pi:X\to B$ be a submersion with closed, oriented and spin$^c$ fibers of even dimension, equipped with a Riemannian and differential spin$^c$ structure $(T^HX, g^{T^VX}, g^\lambda, \nabla^\lambda)$. Let $(E, g^E, \nabla^E)$ be a complex vector bundle over $X$ with a Hermitian metric and a unitary connection. Denote by $\DD^{S\otimes E}$ the corresponding twisted spin$^c$ Dirac operator. Mi\v s\v cenko--Fomenko show in \cite{MF79} that there always exists a $\Z_2$-graded complex vector bundle $L\to B$ such that it represents the analytic index of $\DD^{S\otimes E}$ in $K$-theory, i.e. $\ind^a([E])=[L^+]-[L^-]\in K(B)$. The corresponding local family index theorem (FIT) for $\DD^{S\otimes E}$ by Freed--Lott \cite[(7.26)]{FL10} states that
$$d\wh{\eta}^E(g^E, \nabla^E, T^HX, g^{T^VX}, g^\lambda, \nabla^\lambda, L)=\int_{X/B}\todd(\nabla^{T^VX})\wedge\ch(\nabla^E)-\ch(\nabla^L),$$
where $\nabla^L$ is the projected $\Z_2$-graded unitary connection on $L\to B$ and $\wh{\eta}^E(g^E, \nabla^E, T^HX, g^{T^VX}, g^\lambda, \nabla^\lambda, L)$ is the corresponding Bismut--Cheeger eta form.

The notation $\wh{\eta}^E(g^E, \nabla^E, T^HX, g^{T^VX}, g^\lambda, \nabla^\lambda, L)$ is to indicate the dependence on the geometric structures. Given two sets of geometric structures, denoted by subscripts 0 and 1, variational formula for the Bismut--Cheeger eta form is an explicit expression of the difference
$$\wh{\eta}^E(g^E_1, \nabla^E_1, T^H_1X, g^{T^VX}_1, g^\lambda_1, \nabla^\lambda_1, L_1)-\wh{\eta}^E(g^E_0, \nabla^E_0, T^H_0X, g^{T^VX}_0, g^\lambda_0, \nabla^\lambda_0, L_0)$$
in terms of the geometric structures involved.

In recent years, variational formulas for the Bismut--Cheeger eta form have been established in various settings and have found numerous applications in local index theory. For instance, Liu proves a variational formula for the Bismut--Cheeger eta form in the equivariant setting \cite[Theorem 1.4]{L23} (see also \cite[Theorem 1.2]{L17}) and uses it to prove its functoriality \cite[Theorem 1.6]{L23} (see also \cite[Theorem 1.3]{L17}). On the other hand, we prove a variational formula for the Bismut--Cheeger eta form without the kernel bundle assumption in the even dimensional fiber case \cite[Proposition 1]{H20} and use it to prove the $\Z_2$-graded version of the real part of the Riemann--Roch--Grothendieck (RRG) theorem for complex flat vector bundles in the same case at the differential form level \cite[Theorem 1]{H20}.

The main result of this paper is an extension of \cite[Proposition 1]{H20}, in the sense that the spin$^c$ Dirac operators are allowed to be twisted by isomorphic vector bundles, i.e. an explicit expression for the difference
$$\wh{\eta}^F(g^F, \nabla^F, T^H_1X, g^{T^VX}_1, g^\lambda_1, \nabla^\lambda_1, L_F)-\wh{\eta}^E(g^E, \nabla^E, T^H_0X, g^{T^VX}_0, g^\lambda_0, \nabla^\lambda_0, L_E)$$
in terms of the geometric structures involved, where the complex vector bundles $E\to X$ and $F\to X$ are isomorphic. To establish the main result of this paper, we first prove the following special case.
\begin{prop}\label{prop 1.1}($=$ Proposition \ref{prop 3.2})
Let $\pi:X\to B$ be a submersion with closed, oriented and spin$^c$ fibers of even dimension, equipped with a Riemannian and differential spin$^c$ structure $(T^HX, g^{T^VX}, g^\lambda, \nabla^\lambda)$. Let $(E, g^E)$ be a Hermitian bundle and $(F, g^F, \nabla^F)$ a Hermitian bundle with a unitary connection. If there exists an isometric isomorphism $\alpha:(E, g^E)\to(F, g^F)$ and $L_{E, \alpha}\to B$ is a $\Z_2$-graded complex vector bundle representing the analytic index of $\DD^{S\otimes E, \alpha}$ (defined in terms of $\alpha^*\nabla^F$) in $K$-theory, then there exist a unique $\Z_2$-graded complex vector bundle $L_{F, \alpha}\to B$ and a unique $\Z_2$-graded isometric isomorphism $\wt{\alpha}^L:(L_{E, \alpha}, g^{L_{E, \alpha}})\to(L_{F, \alpha}, g^{L_{F, \alpha}})$ such that $L_{F, \alpha}\to B$ represents the analytic index of $\DD^{S\otimes F}$ (defined in terms of $\nabla^F$) in $K$-theory and
$$\wh{\eta}^E(g^E, \alpha^*\nabla^F, T^HX, g^{T^VX}, g^\lambda, \nabla^\lambda, L_{E, \alpha})=\wh{\eta}^F(g^F, \nabla^F, T^HX, g^{T^VX}, g^\lambda, \nabla^\lambda, L_{F, \alpha}).$$
\end{prop}

The main result of this paper is the following theorem.
\begin{thm}\label{thm 1.1}($=$ Theorem \ref{thm 3.1})
Let $\pi:X\to B$ be a submersion with closed, oriented and spin$^c$ fibers of even dimension, and $(E, g^E, \nabla^E)$ and $(F, g^F, \nabla^F)$ are Hermitian bundles with unitary connections. Denote by $\DD^{S\otimes E}$ and $\DD^{S\otimes F}$ the twisted spin$^c$ Dirac operators defined in terms of the following Riemannian and differential spin$^c$ structures
$$(T^H_0X, g^{T^VX}_0, g^\lambda_0, \nabla^\lambda_0)\qquad\textrm{ and }\qquad(T^H_1X, g^{T^VX}_1, g^\lambda_1, \nabla^\lambda_1)$$
on $\pi:X\to B$, respectively, where the underlying topological spin$^c$ structures coincide. Let $L_E\to B$ and $L_F\to B$ be $\Z_2$-graded complex vector bundles that represent the analytic indexes of $\DD^{S\otimes E}$ and $\DD^{S\otimes F}$ in $K$-theory, respectively.

If there exists an isometric isomorphism $\alpha:(E, g^E)\to(F, g^F)$, then there exist balanced $\Z_2$-graded triples $(W_0, g^{W_0}, \nabla^{W_0})$ and $(W_1, g^{W_1}, \nabla^{W_1})$ and a $\Z_2$-graded isometric isomorphism $h:(L_E\oplus W_0, g^{L_E}\oplus g^{W_0})\to(L_F\oplus W_1, g^{L_F}\oplus g^{W_1})$ such that
\begin{displaymath}
\begin{split}
&\wh{\eta}^F(g^F, \nabla^F, T^H_1X, g^{T^VX}_1, g^\lambda_1, \nabla^\lambda_1, L_F)-\wh{\eta}^E(g^E, \nabla^E, T^H_0X, g^{T^VX}_0, g^\lambda_0, \nabla^\lambda_0, L_E)\\
&=\int_{X/B}\big(T\wh{A}(\nabla^{T^VX}_0, \nabla^{T^VX}_1)\wedge e^{\frac{1}{2}c_1(\nabla^\lambda_0)}
+\wh{A}(\nabla^{T^VX}_1)\wedge e^{\frac{1}{2}Tc_1(\nabla^\lambda_0, \nabla^\lambda_1)}\big)\wedge\ch(\nabla^E)\\
&\quad+\int_{X/B}\todd(\nabla^{T^VX}_1)\wedge\CS(\nabla^E, \alpha^*\nabla^F)-\CS(\nabla^{L_E}\oplus\nabla^{W_0}, h^*(\nabla^{L_F}\oplus\nabla^{W_1}))
\end{split}
\end{displaymath}
in \dis{\frac{\Omega^{\odd}(B)}{\im(d)}}.
\end{thm}

Note that Theorem \ref{thm 1.1} is a special case of the variational formula of the equivariant Bismut--Cheeger eta form by Liu \cite[Theorem 3.17]{L21}, which is proved in the setting of an equivariant version of Bunke--Schick's model of differential $K$-theory \cite{BS09}. A notable difference is the appearance of the equivariant higher spectral flow in \cite[Theorem 3.17]{L21} and the Chern--Simons form in Theorem \ref{thm 1.1}.

We present some applications of Proposition \ref{prop 1.1} and Theorem \ref{thm 1.1} in this paper. We use Proposition \ref{prop 1.1} to establish the $\Z_2$-graded additivity of the Bismut--Cheeger eta form.
\begin{thm}\label{thm 1.2}($=$ Theorem \ref{thm 4.1})
Let $\pi:X\to B$ be a submersion with closed, oriented and spin$^c$ fibers of even dimension, equipped with a Riemannian and differential spin$^c$ structure $(T^HX, g^{T^VX}, g^\lambda, \nabla^\lambda)$. Let $(E, g^E, \nabla^E)$ be a $\Z_2$-graded Hermitian bundle with a $\Z_2$-graded unitary connection. Denote by $\DD^{S\otimes E}$, $\DD^{S\otimes E^+}$ and $\DD^{S\otimes E^-}$ the spin$^c$ Dirac operators twisted by $E\to X$, $E^+\to X$ and $E^-\to X$, respectively. If $L_{E^+}\to B$ and $L_{E^-}\to B$ are $\Z_2$-graded complex vector bundles representing the analytic indexes of $\DD^{S\otimes E^+}$ and $\DD^{S\otimes E^-}$ in $K$-theory, respectively, then $L_{E^+}\oplus L_{E^-}^{\op}\to B$ represents the analytic index of $\DD^{S\otimes E}$ in $K$-theory and
\begin{displaymath}
\begin{split}
&\wh{\eta}^E(g^E, \nabla^E, T^HX, g^{T^VX}, g^\lambda, \nabla^\lambda, L_{E^+}\oplus L_{E^-}^{\op})\\
&=\wh{\eta}^{E^+}(g^{E, +}, \nabla^{E, +}, T^HX, g^{T^VX}, g^\lambda, \nabla^\lambda, L_{E^+})-\wh{\eta}^{E^-}(g^{E, -}, \nabla^{E, -}, T^HX, g^{T^VX}, g^\lambda, \nabla^\lambda, L_{E^-})
\end{split}
\end{displaymath}
in \dis{\frac{\Omega^{\odd}(B)}{\im(d)}}, where $L_{E^-}^{\op}\to B$ denotes the $\Z_2$-graded complex vector bundle with the opposite grading of $L_{E^-}\to B$.
\end{thm}

While the additivity of the Bismut--Cheeger eta form is well known for ungraded direct sums (at least in the case when the Dirac operators satisfy the kernel bundle assumption), its $\Z_2$-graded additivity is less well known. We could not find a statement nor a proof of this result in the literature. For the sake of completeness, we also give a proof of the additivity of the Bismut--Cheeger eta form for ungraded direct sums without the kernel bundle assumption (Proposition \ref{prop 4.1}).

The first application of Theorem \ref{thm 1.1} concerns the analytic index in differential $K$-theory. Given a submersion $\pi:X\to B$ with closed, oriented and spin$^c$ fibers of even dimension equipped with a Riemannian and differential spin$^c$ structure $(T^HX, g^{T^VX}, g^\lambda, \nabla^\lambda)$, the analytic index in differential $K$-theory \cite[Definition 7.27]{FL10} is defined to be
$$\ind^a_{\wh{K}}(\E; L)=\bigg(L, g^L, \nabla^L, \int_{X/B}\todd(\nabla^{T^VX})\wedge\omega+\wh{\eta}^E(g^E, \nabla^E, T^HX, g^{T^VX}, g^\lambda, \nabla^\lambda, L)\bigg),$$
where $\E$ is a generator of the Freed--Lott differential $K$-group $\wh{K}_{\FL}(X)$ and $L\to B$ is a $\Z_2$-graded complex vector bundle representing the analytic index of $\DD^{S\otimes E}$ in $K$-theory. We use Theorem \ref{thm 1.1} and Proposition \ref{prop 4.1} to prove the following result.
\begin{prop}\label{prop 1.2}
Let $\pi:X\to B$ be a submersion with closed, oriented and spin$^c$ fibers of even dimension, equipped with a Riemannian and differential spin$^c$ structure. The analytic index in differential $K$-theory
$$\ind^a_{\wh{K}}:\wh{K}_{\FL}(X)\to\wh{K}_{\FL}(B)$$
is a well defined group homomorphism.
\end{prop}

Proposition \ref{prop 1.2} is first derived by Freed--Lott \cite[(2) of Corollary 7.36]{FL10} as a consequence of the fact that the topological index in differential $K$-theory $\ind^t_{\wh{K}}:\wh{K}_{\FL}(X)\to\wh{K}_{\FL}(B)$ is a well defined group homomorphism \cite[(2) of Proposition 4.18 and Lemma 5.30]{FL10} and the FIT in differential $K$-theory \cite[Theorem 7.35]{FL10}, i.e. for every generator $\E$ of $\wh{K}_{\FL}(X)$,
$$\ind^a_{\wh{K}}(\E)=\ind^t_{\wh{K}}(\E).$$
Our proof of Proposition \ref{prop 1.2} does not rely on these results. It is worth noting that the construction of the topological index in differential $K$-theory is complicated, and the proofs of the well definedness of the topological index $\ind^t_{\wh{K}}$ and the FIT in differential $K$-theory are highly nontrivial.

Since the analytic index in differential $K$-theory $\ind^a_{\wh{K}}:\wh{K}_{\FL}(X)\to\wh{K}_{\FL}(B)$ restricts to the analytic index in $\R/\Z$ $K$-theory \cite[Definition 14]{L94}
\begin{equation}\label{eq 1.0.1}
\ind^a_{\R/\Z}:K^{-1}_{\LL}(X)\to K^{-1}_{\LL}(B),
\end{equation}
as an immediate consequence of Proposition \ref{prop 1.2}, (\ref{eq 1.0.1}) is also a well defined group homomorphism.

Finally, we use Theorems \ref{thm 1.1} and \ref{thm 1.2} to give an alternative proof of the RRG theorem in $\R/\Z$ $K$-theory.
\begin{thm}\label{thm 1.3}
Let $\pi:X\to B$ be a submersion with closed, oriented and spin$^c$ fibers of even dimension, equipped with a Riemannian and differential spin$^c$ structure. The following diagram commutes.
\begin{equation}\label{eq 1.0.2}
\begin{tikzcd}
K^{-1}_{\LL}(X) \arrow{r}{\ch_{\R/\Q}} \arrow{d}[swap]{\ind^a_{\R/\Z}} & H^{\odd}(X; \R/\Q) \arrow{d}{\int_{X/B}\todd(T^VX)\cup(\cdot)} \\ K^{-1}_{\LL}(B) \arrow{r}[swap]{\ch_{\R/\Q}} & H^{\odd}(B; \R/\Q)
\end{tikzcd}
\end{equation}
That is, for every $\Z_2$-graded generator $\E$ of $K^{-1}_{\LL}(X)$,
\begin{equation}\label{eq 1.0.3}
\ch_{\R/\Q}(\ind^a_{\R/\Z}(\E))=\int_{X/B}\todd(T^VX)\cup\ch_{\R/\Q}(\E).
\end{equation}
\end{thm}

Here, $\ch_{\R/\Q}$ is the Chern character in $\R/\Z$ $K$-theory. We prove (\ref{eq 1.0.3}) at the differential form level. Note that (\ref{eq 1.0.3}) implies the commutativity of diagram (\ref{eq 1.0.2}) only if (\ref{eq 1.0.1}) is well defined.

The paper is organized as follows. In \S\ref{s 2.1} we set and fix the notations and conventions used throughout the paper, and in \S\ref{s 2.2} we recall the definitions and properties of some primary and secondary characteristic forms. In \S\ref{s 3.1} we review the local FIT for twisted spin$^c$ Dirac operators without the kernel bundle assumption by the Mi\v s\v cenko--Fomenko--Freed--Lott approach, and in \S\ref{s 3.2} we prove the extended variational formula for the Bismut--Cheeger eta form (Theorem \ref{thm 1.1}). In \S\ref{s 4.1}, we establish some intermediate results on the Bismut--Cheeger eta form, which, together with Theorem \ref{thm 1.1}, allow us to prove the $\Z_2$-graded additivity of the Bismut--Cheeger eta form (Theorem \ref{thm 1.2}). Furthermore, we prove the analytic index in differential $K$-theory is a well defined group homomorphism (Proposition \ref{prop 1.2}) in \S\ref{s 4.2} and the RRG theorem in $\R/\Z$ $K$-theory (Theorem \ref{thm 1.3}) in \S\ref{s 4.3}.

\section*{Acknowledgements}

The author would like to thank Steve Rosenberg for his comments and suggestions for this paper, and Jonathan Kin-Yue Lee, where the idea of Proposition \ref{prop 4.2} is due to him. The author would also like to thank the referee for the helpful comments.

\section{Preliminaries}\label{s 2}

\subsection{Notations and conventions}\label{s 2.1}

In this paper, $X$ and $B$ are closed manifolds and $I$ is the closed interval $[0, 1]$. Given a manifold $X$, define $\wt{X}=X\times I$. Given $t\in I$, define a map $i_{X, t}:X\to\wt{X}$ by $i_{X, t}(x)=(x, t)$. Denote by $p_X:\wt{X}\to X$ the standard projection map. For $k\geq 0$, denote by $\Omega^k_A(X)$ the set of all real-valued closed $k$-forms on $X$ with periods in $A$, where $A$ is a proper subgroup of $\R$. For any differential forms $\omega$ and $\eta$, we write $\omega\equiv\eta$ if $\omega-\eta\in\im(d)$.

Let $\pi:M\to B$ a smooth fiber bundle with compact fibers of dimension $n$ which satisfies certain orientability assumptions. Then
$$\int_{M/B}\pi^*\alpha\wedge\beta=\alpha\wedge\bigg(\int_{M/B}\beta\bigg)$$
for any $\alpha\in\Omega^\bullet(B)$ and $\beta\in\Omega^\bullet(M)$. If $M$ has nonempty boundary, then Stokes' theorem for integration along the fibers \cite[Problem 4 of Chapter VII]{GHV} states that for any $\omega\in\Omega^k(M)$,
\begin{equation}\label{eq 2.1.1}
(-1)^{k-n+1}\int_{\partial M/B}i^*\omega=\int_{M/B}d^M\omega-d^B\int_{M/B}\omega,
\end{equation}
where $i:\partial M\hto M$ is the inclusion map .

Let $E\to X$ be a complex vector bundle. If $E\to X$ is $\Z_2$-graded, denote by $E^{\op}\to X$ the $\Z_2$-graded complex vector bundle whose $\Z_2$-grading is the opposite to that of $E\to X$. We will also use the notation $\op$ for other $\Z_2$-graded objects. A triple $(E, g^E, \nabla^E)$ consisting of a complex vector bundle with a Hermitian metric and a unitary connection is said to be $\Z_2$-graded if $E\to X$ is $\Z_2$-graded and $g^E$ and $\nabla^E$ preserve the $\Z_2$-grading (which are also said to be $\Z_2$-graded). A $\Z_2$-graded triple $(E, g^E, \nabla^E)$ is said to be balanced if $E^+=E^-$, $g^{E, +}=g^{E, -}$ and $\nabla^{E, +}=\nabla^{E, -}$.
\begin{remark}\label{remark 2.1}
Let $F\to X$ be another complex vector bundle. Suppose there exists a smooth bundle isomorphism $\alpha:E\to F$.
\begin{enumerate}
  \item We use the same symbol to denote the resulting $C^\infty(X)$-module isomorphism $\Gamma(X, E)\to\Gamma(X, F)$ and some others, for example, $\Gamma(X, T^*X\otimes E)\to\Gamma(X, T^*X\otimes F)$.
  \item Let $g^E$ and $g^F$ be Hermitian metrics on $E\to X$ and $F\to X$, respectively. Since $g^E$ and $\alpha^*g^F$ are Hermitian metrics on $E\to X$, it follows from \cite[Theorem 8.8 of Chapter I]{K08} that there exists a unique $f\in\Aut(E)$ such that $g^E=f^*(\alpha^*g^F)=(\alpha\circ f)^*g^F$.

      Henceforth, once Hermitian metrics are put on $E\to X$ and $F\to X$, we always assume a given smooth bundle isomorphism $\alpha:E\to F$ is isometric.
  \item Let $\nabla^F$ be a connection on $F\to X$. Write $\alpha^*\nabla^F$ for the connection on $E\to X$ defined by
        \begin{equation}\label{eq 2.1.2}
        \alpha^*\nabla^F:=\alpha^{-1}\circ\nabla^F\circ\alpha.
        \end{equation}
        If $\nabla^F$ is compatible with $g^F$, it follows from our convention that $\alpha^*\nabla^F$ is compatible with $g^E$.
\end{enumerate}
\end{remark}

\subsection{Some primary and secondary characteristic forms}\label{s 2.2}

In this subsection we recall the definitions and properties of some primary and secondary characteristic forms. We refer the readers to \cite[\S B.5]{MM07} for the details.

Let $(E, g^E, \nabla^E)$ be a triple. Denote by $R^E$ the curvature of $\nabla^E$. The Chern character form of $\nabla^E$ is defined to be
$$\ch(\nabla^E)=\tr(e^{-\frac{1}{2\pi i}R^E})\in\Omega^{\even}_\Q(X),$$
and the first Chern form of $\nabla^E$ is defined to be
$$c_1(\nabla^E)=-\frac{1}{2\pi i}\tr(R^E)\in\Omega^2_\Z(X).$$

Let $(g^E_0, \nabla^E_0)$ and $(g^E_1, \nabla^E_1)$ be two pairs of Hermitian metrics and unitary connections on $E\to X$. By (b) of Remark \ref{remark 2.1} there exists a unique $f\in\Aut(E)$ such that $g^E_0=f^*g^E_1$. Thus $f^*\nabla^E_1$ is unitary with respect to $g^E_0$. For each $t\in I$, let $f_t=(1-t)\id_E+tf^{-1}$. Then
\begin{equation}\label{eq 2.2.1}
g^E_t=f_t^*g^E_0,\qquad\nabla^E_t=f_t^*\big((1-t)\nabla^E_0+tf^*\nabla^E_1\big)
\end{equation}
are smooth paths of Hermitian metrics and unitary connections from $g^E_0$ to $g^E_1$ and from $\nabla^E_0$ to $\nabla^E_1$, respectively. Note that $\nabla^E_t$ is unitary with respect to $g^E_t$ for each $t\in I$. Define a complex vector bundle $\e\to\wt{X}$ by $\e=p_X^*E$. Then
\begin{equation}\label{eq 2.2.2}
g^\e:=p_X^*g^E_t,\qquad\nabla^\e:=dt\wedge\bigg(\frac{\partial}{\partial t}+\frac{1}{2}(g^E_t)^{-1}\frac{\partial}{\partial t}g^E_t\bigg)+\nabla^E_t
\end{equation}
are a Hermitian metric and a unitary connection on $\e\to\wt{X}$, respectively, satisfying $i_{X, j}^*\nabla^\e=\nabla^E_j$ for $j\in\set{0, 1}$. The Chern--Simons form \dis{\CS(\nabla^E_0, \nabla^E_1)\in\frac{\Omega^{\odd}(X)}{\im(d)}} is defined by
$$\CS(\nabla^E_0, \nabla^E_1):=-\int_{\wt{X}/X}\ch(\nabla^\e)\mod\im(d).$$
Note that $\CS(\nabla^E_0, \nabla^E_1)$ does not depend on the choice of $\nabla^\e$ satisfying $i_{X, j}^*\nabla^\e=\nabla^E_j$ for $j\in\set{0, 1}$ and satisfies the following transgression formula
$$d\CS(\nabla^E_0, \nabla^E_1)=\ch(\nabla^E_1)-\ch(\nabla^E_0).$$
Equivalently, the Chern--Simons form can be defined as
\begin{equation}\label{eq 2.2.3}
\CS(\nabla^E_0, \nabla^E_1)=-\frac{1}{2\pi i}\int^1_0\tr\bigg(\frac{d\nabla^E_t}{dt}e^{-\frac{1}{2\pi i}R^E_t}\bigg)dt.
\end{equation}
The choices of 0 and 1 are immaterial. If $t<T$ are two fixed positive real numbers, then one can replace 0 by $t$ and 1 by $T$ in (\ref{eq 2.2.3}).

One can also define a differential form \dis{Tc_1(\nabla^E_0, \nabla^E_1)\in\frac{\Omega^1(X)}{\im(d)}} that satisfies the following transgression formula
$$dTc_1(\nabla^E_0, \nabla^E_1)=c_1(\nabla^E_1)-c_1(\nabla^E_0).$$

Let $(E, g^E_k, \nabla^E_k)$ and $(F, g^F_k, \nabla^F_k)$ be triples, where $k\in\set{0, 1}$. The Chern--Simons form satisfies the following properties:
\begin{align}
\CS(\nabla^E_1, \nabla^E_0)&\equiv-\CS(\nabla^E_0, \nabla^E_1),\label{eq 2.2.4}\\
\CS(\nabla^E_1, \nabla^E_0)&\equiv\CS(\nabla^E_1, \nabla^E_2)+\CS(\nabla^E_2, \nabla^E_0),\label{eq 2.2.5}\\
\CS(\nabla^E_1\oplus\nabla^F_1, \nabla^E_0\oplus\nabla^F_0)&\equiv\CS(\nabla^E_1, \nabla^E_0)+\CS(\nabla^F_1, \nabla^F_0).\label{eq 2.2.6}
\end{align}
If $(E, g^E)$ is a $\Z_2$-graded Hermitian bundle with two $\Z_2$-graded unitary connections $\nabla^E_0$ and $\nabla^E_1$, then
\begin{equation}\label{eq 2.2.7}
\CS(\nabla_0^E, \nabla^E_1)\equiv\CS(\nabla_0^{E, +}, \nabla_1^{E, +})-\CS(\nabla_0^{E, -}, \nabla_1^{E, -}).
\end{equation}

Let $(H, g^H, \nabla^H)$ be a Euclidean bundle with a Euclidean connection. Denote by $R^H$ the curvature of $\nabla^H$. The $\wh{A}$-genus form of $\nabla^H$ is defined to be
$$\wh{A}(\nabla^H)=\sqrt{\det\bigg(\frac{-\frac{1}{4\pi i}R^H}{\sinh(-\frac{1}{4\pi i}R^H)}\bigg)}\in\Omega^{4\bullet}_\Q(X).$$
Similarly, one can define a differential form \dis{T\wh{A}(\nabla^H_0, \nabla^H_1)\in\frac{\Omega^{4\bullet-1}(X)}{\im(d)}} that satisfies the following transgression formula
$$dT\wh{A}(\nabla^H_0, \nabla^H_1)=\wh{A}(\nabla^H_1)-\wh{A}(\nabla^H_0).$$

\section{An extended variational formula for the Bismut--Cheeger eta form}\label{s 3}

\subsection{Local index theory for twisted spin$^c$ Dirac operators: the Mi\v s\v cenko--Fomenko--Freed--Lott approach}\label{s 3.1}

In this subsection we review the statement of the local FIT for twisted spin$^c$ Dirac operators without the kernel bundle assumption by Freed--Lott. We refer the readers to \cite[Chapter 10]{BGV} and \cite[\S7]{FL10} for the details.

Let $\pi:X\to B$ be a submersion with closed, oriented and spin$^c$ fibers $Z$ of even dimension $n$. Denote by $T^VX\to X$ the vertical tangent bundle. Recall from \cite[p.918]{FL10} that a Riemannian structure on $\pi:X\to B$ consists of a horizontal distribution $T^HX\to X$, i.e. $TX=T^VX\oplus T^HX$, and a metric $g^{T^VX}$ on $T^VX\to X$. Denote by $P^{T^VX}:TX\to T^VX$ the projection map. Put a Riemannian metric $g^{TB}$ on $TB\to B$. Define a metric $g^{TX}$ on $TX\to X$ by
$$g^{TX}=g^{T^VX}\oplus\pi^*g^{TB}.$$
Denote by $\nabla^{TX}$ and $\nabla^{TB}$ the Levi-Civita connections on $TX\to X$ and $TB\to X$ associated to $g^{TX}$ and $g^{TB}$, respectively. Then $\nabla^{T^VX}:=P^{T^VX}\nabla^{TX}$ is a Euclidean connection on $T^VX\to X$ with respect to $g^{T^VX}$.

Define a connection $\wt{\nabla}^{TX}$ on $TX\to X$ by
$$\wt{\nabla}^{TX}=\nabla^{T^VX}\oplus\pi^*\nabla^{TB}.$$
Then $S:=\nabla^{TX}-\wt{\nabla}^{TX}\in\Omega^1(X, \End(TX))$. By \cite[Theorem 1.9]{B86} the $(3, 0)$ tensor $g^{TX}(S(\cdot)\cdot, \cdot)$ depends only on the Riemannian structure $(T^HX, g^{T^VX})$. Let $\set{e_1, \ldots, e_n}$ be a local orthonormal frame for $T^VX\to X$. For any $U\in\Gamma(B, TB)$, denote by $U^H\in\Gamma(X, T^HX)$ its horizontal lift. Define a horizontal one-form $k$ on $X$ by
\begin{equation}\label{eq 3.1.1}
k(U^H)=-\sum_{k=1}^ng^{TX}(S(e_k)e_k, U^H).
\end{equation}
For any two $U, V\in\Gamma(B, TB)$,
\begin{equation}\label{eq 3.1.2}
T(U, V):=-P^{T^VX}[U^H, V^H]
\end{equation}
is a horizontal two-form with values in $T^VX$ and is called the curvature of $\pi:X\to B$.

Denote by $d\vol(Z)$ the Riemannian volume element of the fiber $Z$, which is a section of $\Lambda^n(T^VX)^*\to X$.

Choose and fix a topological spin$^c$ structure on $T^VX\to X$. This fixes a complex line bundle $\lambda\to X$ satisfying $w_2(T^VX)=c_1(\lambda)\mod 2$ \cite[p.397]{LM89}. The spinor bundle $S(T^VX)\to X$ associated to the chosen topological spin$^c$ structure of $T^VX\to X$ is given by
$$S(T^VX)=S_0(T^VX)\otimes\lambda^{\frac{1}{2}},$$
where $S_0(T^VX)$ is the spinor bundle for the locally existing spin structure of $T^VX\to X$ and $\lambda^{\frac{1}{2}}$ is the locally existing square root of $\lambda\to X$. Since $n$ is even, $S(T^VX)\to X$ is $\Z_2$-graded. Recall from \cite[p.918]{FL10} that a differential spin$^c$ structure on $\pi:X\to B$ consists of a topological spin$^c$ structure on $T^VX\to X$, a Hermitian metric $g^\lambda$ and a unitary connection $\nabla^\lambda$ on $\lambda\to X$.

A Riemannian and differential spin$^c$ structure $(T^HX, g^{T^VX}, g^\lambda, \nabla^\lambda)$ on $\pi:X\to B$ induce a Hermitian metric $g^{S(T^VX)}$ and a unitary connection $\nabla^{S(T^VX)}$ on $S(T^VX)\to X$. Define the Todd form of $\nabla^{T^VX}$ by
\begin{equation}\label{eq 3.1.3}
\todd(\nabla^{T^VX})=\wh{A}(\nabla^{T^VX})\wedge e^{\frac{1}{2}c_1(\nabla^\lambda)}.
\end{equation}

Let $(E, g^E, \nabla^E)$ be a triple. Define the twisted spin$^c$ Dirac operator $\DD^{S\otimes E}:\Gamma(X, S(T^VX)\otimes E)\to\Gamma(X, S(T^VX)\otimes E)$ by
\begin{equation}\label{eq 3.1.4}
\DD^{S\otimes E}=\sum_{k=1}^nc(e_k)\nabla_{e_k}^{S(T^VX)\otimes E},
\end{equation}
where $c$ is the Clifford multiplication and $\nabla^{S(T^VX)\otimes E}$ is the tensor product of $\nabla^{S(T^VX)}$ and $\nabla^E$. Note that $\DD^{S\otimes E}$ is odd self-adjoint.

Define an infinite-rank $\Z_2$-graded complex vector bundle $\pi_*E\to B$ whose fiber over $b\in B$ is given by
$$(\pi_*E)_b=\Gamma(Z_b, (S(T^VX)\otimes E)|_{Z_b}).$$
The space of sections of $\pi_*E\to B$ is defined to be
$$\Gamma(B, \pi_*E):=\Gamma(X, S(T^VX)\otimes E).$$
Define an $L^2$-metric on $\pi_*E\to B$ by
\begin{equation}\label{eq 3.1.5}
g^{\pi_*E}(s_1, s_2)(b)=\int_{Z_b}g^{S(T^VX)\otimes E}(s_1, s_2)d\vol(Z).
\end{equation}
Define a connection on $\pi_*E\to B$ by
\begin{equation}\label{eq 3.1.6}
\nabla^{\pi_*E}_Us:=\nabla^{S(T^VX)\otimes E}_{U^H}s,
\end{equation}
where $s\in\Gamma(B, \pi_*E)$ and $U\in\Gamma(B, TB)$. Then the connection on $\pi_*E\to B$ defined by
\begin{equation}\label{eq 3.1.7}
\nabla^{\pi_*E, u}:=\nabla^{\pi_*E}+\frac{1}{2}k,
\end{equation}
where $k$ is given by (\ref{eq 3.1.1}), is $\Z_2$-graded and unitary with respect to $g^{\pi_*E}$.

The Bismut superconnection on $\pi_*E\to B$ is defined to be
\begin{equation}\label{eq 3.1.8}
\bbb^E=\DD^{S\otimes E}+\nabla^{\pi_*E, u}-\frac{c(T)}{4},
\end{equation}
where $T$ is given by (\ref{eq 3.1.2}). The rescaled Bismut superconnection is given by
$$\bbb^E_t=\sqrt{t}\DD^{S\otimes E}+\nabla^{\pi_*E, u}-\frac{c(T)}{4\sqrt{t}}.$$
By \cite[Theorem 10.23]{BGV},
\begin{equation}\label{eq 3.1.9}
\lim_{t\to 0}\ch(\bbb^E_t)=\int_{X/B}\todd(\nabla^{T^VX})\wedge\ch(\nabla^E).
\end{equation}

Mi\v s\v cenko--Fomenko \cite[p.96-97]{MF79} (see also \cite[Lemma 7.13]{FL10}) prove that there exist finite rank subbundles $L^\pm\to B$ and complementary closed subbundles $K^\pm\to B$ of $(\pi_*E)^\pm\to B$ such that
\begin{equation}\label{eq 3.1.10}
(\pi_*E)^+=K^+\oplus L^+,\qquad(\pi_*E)^-=K^-\oplus L^-,
\end{equation}
$\DD^{S\otimes E}_+:(\pi_*E)^+\to(\pi_*E)^-$ is block diagonal as a map with respect to (\ref{eq 3.1.10}) and $\DD^{S\otimes E}_+|_{K^+}:K^+\to K^-$ is an isomorphism.

Given $L^\pm\to B$ satisfying the above conditions, we say the $\Z_2$-graded complex vector bundle $L\to B$, defined by $L=L^+\oplus L^-$, satisfies the MF property for $\DD^{S\otimes E}$. If $L\to B$ satisfies the MF property for $\DD^{S\otimes E}$, then the analytic index of $[E]\in K(X)$ is defined to be
$$\ind^a([E])=[L^+]-[L^-]\in K(B).$$
It is proved in \cite[p.96-97]{MF79} that $\ind^a([E])$ does not depend on the choice of $L\to B$ satisfying the MF property for $\DD^{S\otimes E}$.

Let $g^L$ be the $\Z_2$-graded Hermitian metric on $L\to B$ inherited from $g^{\pi_*E}$. Denote by $P:\pi_*E\to L$ the $\Z_2$-graded projection map with respect to (\ref{eq 3.1.10}). Define a connection on $L\to B$ by
\begin{equation}\label{eq 3.1.11}
\nabla^L:=P\circ\nabla^{\pi_*E, u}\circ P.
\end{equation}
Note that $\nabla^L$ is $\Z_2$-graded and compatible with $g^L$. Henceforth, whenever $(L, g^L, \nabla^L)$ is a $\Z_2$-graded triple and $L\to B$ satisfies the MF property for $\DD^{S\otimes E}$, $g^L$ and $\nabla^L$ are obtained as above unless otherwise specified.

Given an $L\to B$ satisfying the MF property for $\DD^{S\otimes E}$, consider the infinite rank $\Z_2$-graded complex vector bundle $\pi_*E\oplus L^{\op}\to B$. Let $i^-:L^-\to(\pi_*E)^-$ be the inclusion map and $z\in\C$. Define a map $\wt{\DD}^{S\otimes E}_+(z):(\pi_*E\oplus L^{\op})^+\to(\pi_*E\oplus L^{\op})^-$ by
\begin{equation}\label{eq 3.1.12}
\wt{\DD}^{S\otimes E}_+(z)=\begin{pmatrix} \DD^{S\otimes E}_+ & z i^- \\ z P^+ & 0 \end{pmatrix}.
\end{equation}
Note that $\wt{\DD}^{S\otimes E}_+(z)$ is invertible for all $z\neq 0$ \cite[Lemma 7.20]{FL10}. Define a map $\wt{\DD}^{S\otimes E}(z):\pi_*E\oplus L^{\op}\to\pi_*E\oplus L^{\op}$ by
$$\wt{\DD}^{S\otimes E}(z):=\begin{pmatrix} 0 & (\wt{\DD}^{S\otimes E}_+(z))^* \\ \wt{\DD}^{S\otimes E}_+(z) & 0 \end{pmatrix}.$$

Define a Bismut superconnection on $\pi_*E\oplus L^{\op}\to B$ by
\begin{equation}\label{eq 3.1.13}
\wh{\bbb}^E=\wt{\DD}^{S\otimes E}(1)+\nabla^{\pi_*E, u}\oplus\nabla^{L, \op}-\frac{c(T)}{4}.
\end{equation}

Choose and fix $a\in(0, 1)$. Let $\alpha:[0, \infty)\to I$ be a smooth function that satisfies $\alpha(t)=0$ for all $t\leq a$ and $\alpha(t)=1$ for all $t\geq 1$. Define a rescaled Bismut superconnection by
$$\wh{\bbb}^E_t=\sqrt{t}\wt{\DD}^{S\otimes E}(\alpha(t))+\nabla^{\pi_*E, u}\oplus\nabla^{L, \op}-\frac{c(T)}{4\sqrt{t}}.$$
Since $\wt{\DD}^{S\otimes E}(\alpha(t))$ is invertible for $t\geq 1$,
\begin{equation}\label{eq 3.1.14}
\lim_{t\to\infty}\ch(\wh{\bbb}^E_t)=0.
\end{equation}
On the other hand, for $t\leq a$, $\wh{\bbb}^E_t$ decouples, i.e.
$$\wh{\bbb}^E_t=\bigg(\sqrt{t}\DD^{S\otimes E}+\nabla^{\pi_*E, u}-\frac{c(T)}{4\sqrt{t}}\bigg)\oplus\nabla^{L, \op}=\bbb^E_t\oplus\nabla^{L, \op}.$$
By (\ref{eq 3.1.9}),
\begin{equation}\label{eq 3.1.15}
\begin{split}
\lim_{t\to 0}\ch(\wh{\bbb}^E_t)&=\lim_{t\to 0}\ch(\bbb^E_t)+\ch(\nabla^{L, \op})\\
&=\int_{X/B}\todd(\nabla^{T^VX})\wedge\ch(\nabla^E)-\ch(\nabla^L).
\end{split}
\end{equation}
The Bismut--Cheeger eta form associated to $\wh{\bbb}^E_t$ is defined to be
\begin{equation}\label{eq 3.1.16}
\wh{\eta}^E(g^E, \nabla^E, T^HX, g^{T^VX}, g^\lambda, \nabla^\lambda, L)=\frac{1}{\sqrt{\pi}}\int^\infty_0\str\bigg(\frac{d\wh{\bbb}^E_t}{dt}e^{-\frac{1}{2\pi i}(\wh{\bbb}^E_t)^2}\bigg)dt.
\end{equation}
By (\ref{eq 3.1.14}) and (\ref{eq 3.1.15}), the local FIT for $\DD^{S\otimes E}$ is
\begin{equation}\label{eq 3.1.17}
d\wh{\eta}^E(g^E, \nabla^E, T^HX, g^{T^VX}, g^\lambda, \nabla^\lambda, L)=\int_{X/B}\todd(\nabla^{T^VX})\wedge\ch(\nabla^E)-\ch(\nabla^L).
\end{equation}

\subsection{A proof of the extended variational formula for the Bismut--Cheeger eta form}\label{s 3.2}

In this subsection we prove an extended variational formula for the Bismut--Cheeger eta form for spin$^c$ Dirac operators twisted by isomorphic Hermitian bundles without the kernel bundle assumption (Theorem \ref{thm 1.1}).

The following proposition is the spin$^c$ analog of \cite[Proposition 1]{H20}.
\begin{prop}\label{prop 3.1}
Let $\pi:X\to B$ be a submersion with closed, oriented and spin$^c$ fibers of even dimension, equipped with two sets of Riemannian and differential spin$^c$ structures
$$(T^H_0X, g^{T^VX}_0, g^\lambda_0, \nabla^\lambda_0),\qquad(T^H_1X, g^{T^VX}_1, g^\lambda_1, \nabla^\lambda_1),$$
where the underlying topological spin$^c$ structures coincide. Let $E\to X$ be a complex vector bundle, $(g^E_0, \nabla^E_0)$ and $(g^E_1, \nabla^E_1)$ are two pairs of Hermitian metrics and unitary connections on $E\to X$. For $j\in\set{0, 1}$, write $\DD^{S\otimes E}_j$ for the twisted spin$^c$ Dirac operator defined in terms of
$$(g^E_j, \nabla^E_j, T^H_jX, g^{T^VX}_j, g^\lambda_j, \nabla^\lambda_j).$$
If $(L_j, g^{L_j}, \nabla^{L_j})$ is a $\Z_2$-graded triple so that $L_j\to B$ satisfies the MF property for $\DD^{S\otimes E}_j$, then there exist balanced $\Z_2$-graded triples $(W_0, g^{W_0}, \nabla^{W_0})$ and $(W_1, g^{W_1}, \nabla^{W_1})$ and a $\Z_2$-graded isometric isomorphism
\begin{equation}\label{eq 3.2.1}
h:(L_0\oplus W_0, g^{L_0}\oplus g^{W_0})\to (L_1\oplus W_1, g^{L_1}\oplus g^{W_1})
\end{equation}
such that
\begin{equation}\label{eq 3.2.2}
\begin{split}
&\wh{\eta}^E(g^E_1, \nabla^E_1, T^H_1X, g^{T^VX}_1, g^\lambda_1, \nabla^\lambda_1, L_1)-\wh{\eta}^E(g^E_0, \nabla^E_0, T^H_0X, g^{T^VX}_0, g^\lambda_0, \nabla^\lambda_0, L_0)\\
&\equiv\int_{X/B}\big(T\wh{A}(\nabla^{T^VX}_0, \nabla^{T^VX}_1)\wedge e^{\frac{1}{2}c_1(\nabla^\lambda_0)}+\wh{A}(\nabla^{T^VX}_1)\wedge
e^{\frac{1}{2}Tc_1(\nabla^\lambda_0, \nabla^\lambda_1)}\big)\wedge\ch(\nabla^E_0)\\
&\quad+\int_{X/B}\todd(\nabla^{T^VX}_1)\wedge\CS(\nabla^E_0, \nabla^E_1)-\CS(\nabla^{L_0}\oplus\nabla^{W_0}, h^*(\nabla^{L_1}\oplus\nabla^{W_1})).
\end{split}
\end{equation}
\end{prop}
\begin{proof}
Since the space of splitting maps is affine, there exists a smooth path of horizontal distributions $\set{T^H_tX\to X}_{t\in I}$ joining $T^H_0X\to X$ and $T^H_1X\to X$. Let $(g^E_t, \nabla^E_t)$ be the smooth path of Hermitian metrics and unitary connections on $E\to X$ joining $(g^E_0, \nabla^E_0)$ and $(g^E_1, \nabla^E_1)$ defined by (\ref{eq 2.2.1}). Define a smooth path $g^{T^VX}_t$ of Euclidean metric joining $g^{T^VX}_0$ and $g^{T^VX}_1$ in a similar way. Then
$$c(t)=(g^E_t, \nabla^E_t, T^H_tX, g^{T^VX}_t, g^\lambda_t, \nabla^\lambda_t),$$
where $t\in I$, is a smooth path joining $c(0)=(g^E_0, \nabla^E_0, T^H_0X, g^{T^VX}_0, g^\lambda_0, \nabla^\lambda_0)$ and $c(1)=(g^E_1, \nabla^E_1, T^H_1X, g^{T^VX}_1, g^\lambda_1, \nabla^\lambda_1)$. Define a new path $\wt{c}(t)$ by
$$\wt{c}(t)=\left\{
              \begin{array}{ll}
                (g^E_0, \nabla^E_0, T^H_{3t}X, g^{T^VX}_{3t}, g^\lambda_0, \nabla^\lambda_0), & \displaystyle\textrm{ for } t\in\bigg[0, \frac{1}{3}\bigg]\\\\
                (g^E_0, \nabla^E_0, T^H_1X, g^{T^VX}_1, g^\lambda_{3t-1}, \nabla^\lambda_{3t-1}), & \displaystyle\textrm{ for } t\in\bigg[\frac{1}{3}, \frac{2}{3}\bigg]\\\\
                (g^E_{3t-2}, \nabla^E_{3t-2}, T^H_1X, g^{T^VX}_1, g^\lambda_1, \nabla^\lambda_1), & \displaystyle\textrm{ for } t\in\bigg[\frac{2}{3}, 1\bigg]
              \end{array}
            \right..$$

Consider the submersion $\wt{\pi}:\wt{X}\to\wt{B}$, where $\wt{\pi}=\pi\times\id_I$. Define complex vector bundles $\e\to\wt{X}$ and $\wt{\lambda}\to\wt{X}$ by $\e=p_X^*E$ and $\wt{\lambda}=p_X^*\lambda$, respectively. Define the pair of Hermitian metric and unitary connection $(g^\e, \nabla^\e)$ on $\e\to\wt{X}$ and the Riemannian and differential spin$^c$ structure $(T^H\wt{X}, g^{T^V\wt{X}}, g^{\wt{\lambda}}, \nabla^{\wt{\lambda}})$ on $\wt{\pi}:\wt{X}\to\wt{B}$ so that for each $t\in I$, the restriction of
\begin{equation}\label{eq 3.2.3}
(g^\e, \nabla^\e, T^H\wt{X}, g^{T^V\wt{X}}, g^{\wt{\lambda}}, \nabla^{\wt{\lambda}})
\end{equation}
to $X\times\set{t}$ is given by $\wt{c}(t)$. Denote by $\DD^{S\otimes\e}$ the twisted spin$^c$ Dirac operator defined in terms of (\ref{eq 3.2.3}). Let $(\LLL, g^\LLL, \nabla^\LLL)$ be a $\Z_2$-graded triple so that $\LLL\to\wt{B}$ satisfies the MF property for $\DD^{S\otimes\e}$, i.e. there exists a $\Z_2$-graded complementary closed subbundle $\K\to\wt{B}$ of $\wt{\pi}_*\e\to\wt{B}$ such that
\begin{equation}\label{eq 3.2.4}
\wt{\pi}_*\e=\K\oplus\LLL,
\end{equation}
$\DD^{S\otimes\e}_+:(\wt{\pi}_*\e)^+\to(\wt{\pi}_*\e)^-$ is block diagonal as a map with respect to (\ref{eq 3.2.4}), and $\DD^{S\otimes\e}_+|_{\K^+}:\K^+\to\K^-$ is a smooth bundle isomorphism. Note that
\begin{equation}\label{eq 3.2.5}
i_{B, 0}^*(\wt{\pi}_*\e)^\pm\cong i_{B, 1}^*(\wt{\pi}_*\e)^\pm=(\pi_*E)^\pm
\end{equation}
and there exist smooth bundle isomorphisms
\begin{equation}\label{eq 3.2.6}
i_{B, 0}^*\LLL^\pm\cong i_{B, 1}^*\LLL^\pm,\qquad i_{B, 0}^*\K^\pm\cong i_{B, 1}^*\K^\pm.
\end{equation}
Let $j\in\set{0, 1}$. Write $\wt{L}_j^\pm\to B$ for $i_{B, j}^*\LLL^\pm\to B$ and $K^\pm_j\to B$ for $i_{B, j}^*\K^\pm\to B$. Moreover, write $g^{\wt{L}_j}$ for $i_{B, j}^*g^\LLL$. Define a $\Z_2$-graded complex vector bundle $\wt{L}_j\to B$ by $\wt{L}_j=\wt{L}_j^+\oplus\wt{L}_j^-$. By (\ref{eq 3.2.6}), we choose and fix a $\Z_2$-graded isometric isomorphism
\begin{equation}\label{eq 3.2.7}
f:(\wt{L}_0, g^{\wt{L}_0})\to(\wt{L}_1, g^{\wt{L}_1}).
\end{equation}

By (\ref{eq 3.2.4}) and (\ref{eq 3.2.5}) we have
\begin{equation}\label{eq 3.2.8}
(\pi_*E)^+=K^+_j\oplus L^+_j,\qquad(\pi_*E)^-=K^-_j\oplus L^-_j.
\end{equation}
Since
$$\DD^{S\otimes\e}|_{i_{B, j}^*\wt{\pi}_*\e}=\DD^{S\otimes E}_j,$$
it follows that $\DD^{S\otimes E}_{j, +}:(\pi_*E)^+\to(\pi_*E)^-$ is block diagonal as a map with respect to (\ref{eq 3.2.8}) and the restriction $\DD^{S\otimes E}_{j, +}|_{K_j^+}:K^+_j\to K^-_j$ is an isomorphism. Thus $\wt{L}_j\to B$ satisfies the MF property for $\DD^{S\otimes E}_j$. Since $L_j\to B$ satisfies the MF property for $\DD^{S\otimes E}_j$ by assumption, there exist balanced $\Z_2$-graded triples $(H_j, g^{H_j}, \nabla^{H_j})$ and $(W_j, g^{W_j}, \nabla^{W_j})$ and a $\Z_2$-graded isometric isomorphism
\begin{equation}\label{eq 3.2.9}
f_j:(\wt{L}_j\oplus H_j, g^{\wt{L}_j}\oplus g^{H_j})\to(L_j\oplus W_j, g^{L_j}\oplus g^{W_j}).
\end{equation}
By (\ref{eq 3.2.7}) and (\ref{eq 3.2.9}), we have the following $\Z_2$-graded isometric isomorphisms:
\begin{equation}\label{eq 3.2.10}
\begin{tikzcd}
(L_0\oplus W_0\oplus H_1, g^{L_0}\oplus g^{W_0}\oplus g^{H_1}) \arrow{d}{f_0^{-1}\oplus\id_{H_1}} \\ (\wt{L}_0\oplus H_0\oplus H_1, g^{\wt{L}_0}\oplus g^{H_0}\oplus g^{H_1}) \arrow{d}{f\oplus\id_{H_0}\oplus\id_{H_1}} \\ (\wt{L}_1\oplus H_0\oplus H_1, g^{\wt{L}_1}\oplus g^{H_0}\oplus g^{H_1}) \arrow{d}{f_1\oplus\id_{H_0}} \\ (L_1\oplus W_1\oplus H_0, g^{L_1}\oplus g^{W_1}\oplus g^{H_0})
\end{tikzcd}
\end{equation}
Write $W_0\to B$ for $W_0\oplus H_1\to B$, $H\to B$ for $H_0\oplus H_1\to B$, $W_1\to B$ for $W_1\oplus H_0\to B$ and similarly for the corresponding $\Z_2$-graded Hermitian metrics and $\Z_2$-graded unitary connections. Moreover, write $f_0$ for $f_0\oplus\id_{H_1}$ and $f_1$ for $f_1\oplus\id_{H_0}$. Then the following diagram commutes.
\begin{center}
\begin{tikzcd}
(\wt{L}_0\oplus H, g^{\wt{L}_0}\oplus g^H) \arrow{r}{f\oplus\id_H} \arrow{d}[swap]{f_0} & (\wt{L}_1\oplus H, g^{\wt{L}_1}\oplus g^H) \arrow{d}{f_1} \\ (L_0\oplus W_0, g^{L_0}\oplus g^{W_0}) \arrow{r}[swap]{h} & (L_1\oplus W_1, g^{L_1}\oplus g^{W_1}),
\end{tikzcd}
\end{center}
where $h:=f_1\circ(f\oplus\id_H)\circ f_0^{-1}$. Furthermore, by writing $L\to B$ for $\wt{L}_0\to B$ and $f_1$ for $f_1\circ(f\oplus\id_H)$, the above diagram becomes
\begin{equation}\label{eq 3.2.11}
\begin{tikzcd}
 & (L\oplus H, g^L\oplus g^H) \arrow{dl}[swap]{f_0} \arrow{dr}{f_1} & \\ (L_0\oplus W_0, g^{L_0}\oplus g^{W_0}) \arrow{rr}[swap]{h} & & (L_1\oplus W_1, g^{L_1}\oplus g^{W_1})
\end{tikzcd}
\end{equation}
and $h$ becomes $h=f_1\circ f_0^{-1}$. Thus (\ref{eq 3.2.1}) holds.

The connection on $L\oplus H\to B$ defined by
\begin{equation}\label{eq 3.2.12}
\nabla^{L\oplus H}_j:=f_j^*(\nabla^{L_j}\oplus\nabla^{W_j})
\end{equation}
is $\Z_2$-graded and unitary. Define a complex vector bundle $\HH\to\wt{B}$ by $\HH=p_B^*H$. Let $(g^{\LLL\oplus\HH}, \nabla^{\LLL\oplus\HH})$ be a pair of Hermitian metric and unitary connection on $\LLL\oplus\HH\to\wt{B}$ defined by (\ref{eq 2.2.2}), which satisfies
$$i_{B, j}^*\nabla^{\LLL\oplus\HH}=\nabla^{L\oplus H}_j.$$
Define a rescaled Bismut superconnection $\wh{\bbb}^{\e; \HH}_t$ on $\wt{\pi}_*\e\oplus(\LLL\oplus\HH)^{\op}\to\wt{B}$ by
$$\wh{\bbb}^{\e; \HH}_t=\sqrt{t}(\wt{\DD}^{S\otimes\e}(\alpha(t))\oplus\alpha(t)s_\HH)+(\nabla^{\wt{\pi}_*\e, u}\oplus\nabla^{\LLL\oplus\HH, \op})-\frac{1}{\sqrt{t}}\bigg(\frac{c(\wt{T})}{4}\oplus 0\bigg),$$
where \dis{s_\HH=\begin{pmatrix} 0 & \id \\ \id & 0 \end{pmatrix}\in\Gamma(\wt{B}, \End(\HH)^-)} and $\wt{T}$ is the curvature 2-form of $\wt{\pi}:\wt{X}\to\wt{B}$. By the definition of $\alpha$, $\wh{\bbb}^{\e; \HH}_t$ decouples for $t\leq a$, i.e.
$$\wh{\bbb}^{\e; \HH}_t=\bbb^\e_t\oplus\nabla^{\LLL\oplus\HH, \op}.$$
By (\ref{eq 3.1.9}) we have
\begin{equation}\label{eq 3.2.13}
\begin{split}
\lim_{t\to 0}\ch(\wh{\bbb}^{\e; \HH}_t)&=\lim_{t\to 0}\ch(\bbb^\e_t)-\ch(\nabla^{\LLL\oplus\HH})\\
&=\int_{\wt{X}/\wt{B}}\todd(\nabla^{T^V\wt{X}})\wedge\ch(\nabla^\e)-\ch(\nabla^{\LLL\oplus\HH}).
\end{split}
\end{equation}
On the other hand, since $(\wh{\bbb}^{\e; \HH}_t)_{[0]}$ is invertible for every $t\geq 1$, it follows that
\begin{equation}\label{eq 3.2.14}
\lim_{t\to\infty}\ch(\wh{\bbb}^{\e; \HH}_t)=0.
\end{equation}
Write $\wh{\eta}^{\e; \HH}(g^\e, \nabla^\e, T^H\wt{X}, g^{T^V\wt{X}}, g^{\wt{\lambda}}, \nabla^{\wt{\lambda}}, \LLL)$ for the Bismut--Cheeger eta form associated to $\wh{\bbb}^{\e; \HH}_t$. We temporarily suppress the data defining the Bismut--Cheeger eta form to shorten the notation. By (\ref{eq 3.2.13}) and (\ref{eq 3.2.14}) we have
\begin{equation}\label{eq 3.2.15}
d\wh{\eta}^{\e; \HH}=\int_{\wt{X}/\wt{B}}\todd(\nabla^{T^V\wt{X}})\wedge\ch(\nabla^\e)-\ch(\nabla^{\LLL\oplus\HH}).
\end{equation}
Denote by $i:\partial\wt{B}\to\wt{B}$ the inclusion map. By (\ref{eq 2.1.1}) we have
\begin{equation}\label{eq 3.2.16}
-(i_{B, 1}^*\wh{\eta}^{\e; \HH}-i_{B, 0}^*\wh{\eta}^{\e; \HH})=-\int_{\partial\wt{B}/B}i^*\wh{\eta}^{\e; \HH}=\int_{\wt{B}/B}d^{\wt{B}}\wh{\eta}^{\e; \HH}-d^B\int_{\wt{B}/B}\wh{\eta}^{\e; \HH}.
\end{equation}
By (\ref{eq 3.2.15}), (\ref{eq 3.2.16}) becomes
\begin{equation}\label{eq 3.2.17}
\begin{split}
i_{B, 1}^*\wh{\eta}^{\e; \HH}-i_{B, 0}^*\wh{\eta}^{\e; \HH}&\equiv-\int_{\wt{B}/B}d^{\wt{B}}\wh{\eta}^{\e; \HH}\\
&\equiv-\int_{\wt{B}/B}\bigg(\int_{\wt{X}/\wt{B}}\todd(\nabla^{T^V\wt{X}})\wedge\ch(\nabla^\e)-\ch(\nabla^{\LLL\oplus\HH})\bigg)\\
&\equiv-\int_{\wt{B}/B}\int_{\wt{X}/\wt{B}}\todd(\nabla^{T^V\wt{X}})\wedge\ch(\nabla^\e)-\CS(\nabla^{L\oplus H}_0, \nabla^{L\oplus H}_1).
\end{split}
\end{equation}
Note that by (\ref{eq 3.2.12}) and the definition of $h$, we have
\begin{equation}\label{eq 3.2.18}
\begin{split}
\CS(\nabla^{L\oplus H}_0, \nabla^{L\oplus H}_1)&=\CS\big(f_0^*(\nabla^{L_0}\oplus\nabla^{W_0}), f_1^*(\nabla^{L_1}\oplus\nabla^{W_1})\big)\\
&\equiv\CS\big(\nabla^{L_0}\oplus\nabla^{W_0}, (f_0^{-1})^*f_1^*(\nabla^{L_1}\oplus\nabla^{W_1})\big)\\
&=\CS\big(\nabla^{L_0}\oplus\nabla^{W_0}, h^*(\nabla^{L_1}\oplus\nabla^{W_1})\big).
\end{split}
\end{equation}
On the other hand, since \dis{\int_{X/B}\circ\int_{\wt{X}/X}=\int_{\wt{X}/B}=\int_{\wt{B}/B}\circ\int_{\wt{X}/\wt{B}}}, it follows from (\ref{eq 3.2.3}) and the definition of $\wt{c}(t)$ that
\begin{equation}\label{eq 3.2.19}
\begin{split}
&-\int_{\wt{B}/B}\int_{\wt{X}/\wt{B}}\todd(\nabla^{T^V\wt{X}})\wedge\ch(\nabla^\e)\\
&=-\int_{X/B}\int_{\wt{X}/X}\todd(\nabla^{T^V\wt{X}})\wedge\ch(\nabla^\e)\\
&\equiv\int_{X/B}\big(T\wh{A}(\nabla^{T^VX}_0, \nabla^{T^VX}_1)\wedge e^{\frac{1}{2}c_1(\nabla^\lambda_0)}+\wh{A}(\nabla^{T^VX}_1)\wedge
e^{\frac{1}{2}Tc_1(\nabla^\lambda_0, \nabla^\lambda_1)}\big)\wedge\ch(\nabla^E_0)\\
&\quad+\int_{X/B}\todd(\nabla^{T^VX}_1)\wedge\CS(\nabla^E_0, \nabla^E_1).
\end{split}
\end{equation}
By (\ref{eq 3.2.18}) and (\ref{eq 3.2.19}), (\ref{eq 3.2.17}) becomes
\begin{equation}\label{eq 3.2.20}
\begin{split}
&i_{B, 1}^*\wh{\eta}^{\e; \HH}-i_{B, 0}^*\wh{\eta}^{\e; \HH}\\
&\equiv\int_{X/B}\big(T\wh{A}(\nabla^{T^VX}_0, \nabla^{T^VX}_1)\wedge e^{\frac{1}{2}c_1(\nabla^\lambda_0)}+\wh{A}(\nabla^{T^VX}_1)\wedge
e^{\frac{1}{2}Tc_1(\nabla^\lambda_0, \nabla^\lambda_1)}\big)\wedge\ch(\nabla^E_0)\\
&\quad+\int_{X/B}\todd(\nabla^{T^VX}_1)\wedge\CS(\nabla^E_0, \nabla^E_1)-\CS\big(\nabla^{L_0}\oplus\nabla^{W_0}, h^*(\nabla^{L_1}\oplus\nabla^{W_1})\big).
\end{split}
\end{equation}
By (\ref{eq 3.2.20}), to prove (\ref{eq 3.2.2}) it remains to show that
\begin{equation}\label{eq 3.2.21}
i_{B, j}^*\wh{\eta}^{\e; \HH}(g^\e, \nabla^\e, T^H\wt{X}, g^{T^V\wt{X}}, g^{\wt{\lambda}}, \nabla^{\wt{\lambda}}, \LLL)\equiv\wh{\eta}^E(g^E_j, \nabla^E_j, T^H_jX, g^{T^VX}_j, g^\lambda_j, \nabla^\lambda_j, L_j)
\end{equation}
for $j\in\set{0, 1}$. First note that $i_{B, j}^*\wh{\bbb}^{\e; \HH}_t=\wh{\bbb}^{E; H}_{j, t}$, where $\wh{\bbb}^{E; H}_{j, t}$ is the rescaled Bismut superconnection on $\pi_*E\oplus(L\oplus H)^{\op}\to B$ given by
\begin{equation}\label{eq 3.2.22}
\wh{\bbb}^{E; H}_{j, t}=\sqrt{t}\big(\wt{\DD}^{S\otimes E}_j(\alpha(t))\oplus\alpha(t)s_H\big)+\big(\nabla^{\pi_*E, u}\oplus\nabla^{L\oplus H, \op}_j\big)-\frac{1}{\sqrt{t}}\bigg(\frac{c(T)}{4}\oplus 0\bigg).
\end{equation}
Here, \dis{s_H:=\begin{pmatrix} 0 & \id \\ \id & 0 \end{pmatrix}\in\Gamma(B, \End(H)^-)}. By (\ref{eq 3.2.12}) and the fact that the $\Z_2$-graded triple $(W_j, g^{W_j}, \nabla^{W_j})$ is balanced, we have
\begin{equation}\label{eq 3.2.23}
\nabla^{L\oplus H, \op}_j=(f_j^{\op})^*(\nabla^{L_j, \op}\oplus\nabla^{W_j}).
\end{equation}
Consider the split quadruple $(W_j, g^{W_j}, \nabla^{W_j}, s_{W_j})$ given in \cite[Example 1]{H20}. Denote by $\aaa_t^{W_j}$ the rescaled superconnection on $W_j\to B$ given by \cite[(3.2)]{H20}. Since $\id_{\pi_*E}\oplus(f_j^{-1})^{\op}:\pi_*E\oplus(L_j\oplus W_j)^{\op}\to\pi_*E\oplus(L\oplus H)^{\op}$ is a $\Z_2$-graded isometric isomorphism, it follows from (\ref{eq 3.2.22}) and (\ref{eq 3.2.23}) that
\begin{displaymath}
\begin{split}
&(\id_{\pi_*E}\oplus(f_j^{-1})^{\op})^*\wh{\bbb}^{E; H}_{j, t}\\
&=\sqrt{t}\big(\wt{\DD}^{S\otimes E}_j(\alpha(t))\oplus\alpha(t)s_{W_j}\big)+\big(\nabla^{\pi_*E, u}\oplus(\nabla^{L_j, \op}\oplus\nabla^{W_j})\big)-\frac{1}{\sqrt{t}}\bigg(\frac{c(T)}{4}\oplus 0\bigg)\\
&=\wh{\bbb}^E_{j, t}\oplus\aaa^{W_j}_t.
\end{split}
\end{displaymath}
Since $\id_{\pi_*E}\oplus(f_j^{-1})^{\op}$ covers $\id_B$, for any $t<T\in(0, \infty)$ we have
\begin{displaymath}
\begin{split}
\CS(i_{B, j}^*\wh{\bbb}^{\e; \HH}_T, i_{B, j}^*\wh{\bbb}^{\e; \HH}_t)&\equiv\CS(\wh{\bbb}^{E; H}_{j, T}, \wh{\bbb}^{E; H}_{j, t})\\
&\equiv\CS\big((\id_{\pi_*E}\oplus(f_j^{-1})^{\op})^*\wh{\bbb}^{E; H}_{j, T}, (\id_{\pi_*E}\oplus(f_j^{-1})^{\op})^*\wh{\bbb}^{E; H}_{j, t}\big)\\
&\equiv\CS(\wh{\bbb}^E_{j, T}\oplus\aaa^{W_j}_T, \wh{\bbb}^E_{j, t}\oplus\aaa^{W_j}_t).
\end{split}
\end{displaymath}
By letting $t\to 0$ and $T\to\infty$ in above, it follows from \cite[Lemma 1]{H20} that
\begin{displaymath}
\begin{split}
i_{B, j}^*\wh{\eta}^{\e; \HH}(g^\e, \nabla^\e, T^H\wt{X}, g^{T^V\wt{X}}, g^{\wt{\lambda}}, \nabla^{\wt{\lambda}}, \LLL)&\equiv\wh{\eta}^{E; W_j}(g^E_j, \nabla^E_j, T^H_jX, g^{T^VX}_j, g^\lambda_j, \nabla^\lambda_j, L_j)\\
&\equiv\wh{\eta}^E(g^E_j, \nabla^E_j, T^H_jX, g^{T^VX}_j, g^\lambda_j, \nabla^\lambda_j, L_j).
\end{split}
\end{displaymath}
Thus (\ref{eq 3.2.2}) holds.
\end{proof}
\begin{remark}\label{remark 3.1}
In Proposition \ref{prop 3.1}, suppose there exists a $\Z_2$-graded isometric isomorphism $\beta:(L_0, g^{L_0})\to(L_1, g^{L_1})$. By taking $h=\beta\oplus\id_W$ and $f_1=h\circ f_0$, diagram (\ref{eq 3.2.11}) still commutes, and (\ref{eq 3.2.18}) becomes
\begin{displaymath}
\begin{split}
\CS\big(f_0^*(\nabla^{L_0}\oplus\nabla^W), f_1^*(\nabla^{L_1}\oplus\nabla^W)\big)&=\CS\big(f_0^*(\nabla^{L_0}\oplus\nabla^W), f_0^*(h^*(\nabla^{L_1}\oplus\nabla^W))\big)\\
&=\CS\big(\nabla^{L_0}\oplus\nabla^W, h^*(\nabla^{L_1}\oplus\nabla^W)\big)\\
&=\CS(\nabla^{L_0}\oplus\nabla^W, \beta^*\nabla^{L_1}\oplus\id_W^*\nabla^W)\\
&\equiv\CS(\nabla^{L_0}, \beta^*\nabla^{L_1}),
\end{split}
\end{displaymath}
where the last equality follows from (\ref{eq 2.2.6}) and the fact that $\CS(\nabla^W, \id_W^*\nabla^W)\equiv 0$.

Thus if $L_0\cong L_1$ as $\Z_2$-graded complex vector bundles in Proposition \ref{prop 3.1}, without loss of generality we can take $W_0\to B$ and $W_1\to B$ to be the zero bundle.
\end{remark}
\begin{prop}\label{prop 3.2}
Let $\pi:X\to B$ be a submersion with closed, oriented and spin$^c$ fibers of even dimension, equipped with a Riemannian and differential spin$^c$ structure $(T^HX, g^{T^VX}, g^\lambda, \nabla^\lambda)$. Let $(E, g^E)$ and $(F, g^F)$ be Hermitian bundles and $\nabla^F$ a unitary connection on $F\to X$. If there exists an isometric isomorphism $\alpha:(E, g^E)\to(F, g^F)$ and $(L_{E, \alpha}, g^{L_{E, \alpha}}, \nabla^{L_{E, \alpha}})$ is a $\Z_2$-graded triple so that $L_{E, \alpha}\to B$ satisfies the MF property for $\DD^{S\otimes E, \alpha}$, which is defined in terms of $\alpha^*\nabla^F$, then there exist a unique $\Z_2$-graded triple $(L_{F, \alpha}, g^{L_{F, \alpha}}, \nabla^{L_{F, \alpha}})$ and a unique $\Z_2$-graded isometric isomorphism $\wt{\alpha}^L:(L_{E, \alpha}, g^{L_{E, \alpha}})\to(L_{F, \alpha}, g^{L_{F, \alpha}})$ such that $L_{F, \alpha}\to B$ satisfies the MF property for $\DD^{S\otimes F}$, which is defined in terms of $\nabla^F$, and
\begin{equation}\label{eq 3.2.24}
\wh{\eta}^E(g^E, \alpha^*\nabla^F, T^HX, g^{T^VX}, g^\lambda, \nabla^\lambda, L_{E, \alpha})=\wh{\eta}^F(g^F, \nabla^F, T^HX, g^{T^VX}, g^\lambda, \nabla^\lambda, L_{F, \alpha}).
\end{equation}
\end{prop}

One can think of Proposition \ref{prop 3.2} as the ``variational" formula for the Bismut--Cheeger eta form for the triples $(E, g^E, \alpha^*\nabla^F)$ and $(F, g^F, \nabla^F)$.
\begin{proof}
Note that $\alpha^*\nabla^F$ is a unitary connection on $E\to X$. By (a) of Remark \ref{remark 2.1},
\begin{align}
\nabla^{S(T^VX)\otimes E}&=\nabla^{T^VX}\otimes\id+\id\otimes\alpha^*\nabla^F\nonumber\\
&=\alpha^*(\nabla^{T^VX}\otimes\id+\id\otimes\nabla^F)\nonumber\\
\Rightarrow\nabla^{S(T^VX)\otimes E}&=\alpha^*\nabla^{S(T^VX)\otimes F}.\label{eq 3.2.25}
\end{align}
Thus
\begin{equation}\label{eq 3.2.26}
\DD^{S\otimes F}=\alpha\circ\DD^{S\otimes E, \alpha}\circ\alpha^{-1}.
\end{equation}

Write $\wt{\alpha}:\pi_*E\to\pi_*F$ for the $\Z_2$-graded bundle isomorphism induced by $\alpha$. Since $g^E=\alpha^*g^F$, it follows from (\ref{eq 3.1.5}) that
\begin{equation}\label{eq 3.2.27}
g^{\pi_*E}=\wt{\alpha}^*g^{\pi_*F}.
\end{equation}
That is, $\wt{\alpha}$ is isometric. By (\ref{eq 3.1.6}) and (\ref{eq 3.2.25}), $\nabla^{\pi_*E}=\wt{\alpha}^*\nabla^{\pi_*F}$, and therefore
\begin{equation}\label{eq 3.2.28}
\nabla^{\pi_*E, u}=\wt{\alpha}^*\nabla^{\pi_*F, u}.
\end{equation}

Let $K_{E, \alpha}\to B$ be a $\Z_2$-graded closed subbundle of $\pi_*E\to B$ that is complementary to $L_{E, \alpha}\to B$, i.e.
\begin{equation}\label{eq 3.2.29}
(\pi_*E)^+=K_{E, \alpha}^+\oplus L_{E, \alpha}^+,\qquad(\pi_*E)^-=K_{E, \alpha}^-\oplus L_{E, \alpha}^-,
\end{equation}
$\DD^{S\otimes E, \alpha}_+:(\pi_*E)^+\to(\pi_*E)^-$ is block diagonal as a map with respect to (\ref{eq 3.2.29}) and $\DD^{S\otimes E, \alpha}_+|_{K_{E, \alpha}^+}:K_{E, \alpha}^+\to K_{E, \alpha}^-$ is an isomorphism. Write \dis{\DD^{S\otimes E, \alpha}_+=\begin{pmatrix} a_{E, \alpha} & 0 \\ 0 & d_{E, \alpha} \end{pmatrix}}. Define complex vector bundles $L_{F, \alpha}^\pm\to B$ and $K_{F, \alpha}^\pm\to B$ by
$$L_{F, \alpha}^\pm:=\wt{\alpha}_\pm(L_{E, \alpha}^\pm),\qquad K_{F, \alpha}^\pm:=\wt{\alpha}_\pm(K_{E, \alpha}^\pm).$$
Denote by $\wt{\alpha}_\pm^L:L_{E, \alpha}^\pm\to L_{F, \alpha}^\pm$ and $\wt{\alpha}_\pm^K:K_{E, \alpha}^\pm\to K_{F, \alpha}^\pm$ the corresponding restriction maps. These are isomorphisms. Note that $L_{F, \alpha}^\pm\to B$ have finite ranks and
\begin{equation}\label{eq 3.2.30}
(\pi_*F)^+=K_{F, \alpha}^+\oplus L_{F, \alpha}^+,\qquad(\pi_*F)^-=K_{F, \alpha}^-\oplus L_{F, \alpha}^-.
\end{equation}
With respect to (\ref{eq 3.2.29}) and (\ref{eq 3.2.30}), $\wt{\alpha}_+$ and $\wt{\alpha}_-$ are given by
\begin{equation}\label{eq 3.2.31}
\wt{\alpha}_+=\begin{pmatrix} \wt{\alpha}_+^K & 0 \\ 0 & \wt{\alpha}_+^L \end{pmatrix},\qquad\wt{\alpha}_-=\begin{pmatrix} \wt{\alpha}_-^K & 0 \\ 0 & \wt{\alpha}_-^L \end{pmatrix}.
\end{equation}
Define $\Z_2$-graded complex vector bundles $L_{F, \alpha}\to B$ and $K_{F, \alpha}\to B$ by
$$L_{F, \alpha}=L_{F, \alpha}^+\oplus L_{F, \alpha}^-\qquad\textrm{ and }\qquad K_{F, \alpha}=K_{F, \alpha}^+\oplus K_{F, \alpha}^-.$$
Then the map $\wt{\alpha}^L:L_{E, \alpha}\to L_{F, \alpha}$ given by
$$\wt{\alpha}^L=\wt{\alpha}^L_+\oplus\wt{\alpha}^L_-$$
is a $\Z_2$-graded smooth bundle isomorphism. Since the Hermitian metric $g^{L_{E, \alpha}}$ on $L_{E, \alpha}\to B$ is inherited from $g^{\pi_*E}$, it follows from (\ref{eq 3.2.27}) that the Hermitian metric $g^{L_{F, \alpha}}$ on $L_{F, \alpha}\to B$, inherited from $g^{\pi_*F}$, satisfies $g^{L_{E, \alpha}}=(\wt{\alpha}^L)^*g^{L_{F, \alpha}}$. Thus $\wt{\alpha}^L$ is isometric.

By (\ref{eq 3.2.26}) and (\ref{eq 3.2.31}),
\begin{displaymath}
\begin{split}
\DD^{S\otimes F}_+&=\wt{\alpha}_-\circ\DD^{S\otimes E, \alpha}_+\circ\wt{\alpha}_+^{-1}\\
&=\begin{pmatrix} \wt{\alpha}_-^K & 0 \\ 0 & \wt{\alpha}_-^L \end{pmatrix}\begin{pmatrix} a_{E, \alpha} & 0 \\ 0 & d_{E, \alpha} \end{pmatrix}\begin{pmatrix} (\wt{\alpha}_+^K)^{-1} & 0 \\ 0 & (\wt{\alpha}_+^L)^{-1} \end{pmatrix}\\
&=\begin{pmatrix} \wt{\alpha}_-^K\circ a_{E, \alpha}\circ(\wt{\alpha}_+^K)^{-1} & 0 \\ 0 & \wt{\alpha}_-^L\circ d_{E, \alpha}\circ(\wt{\alpha}_+^L)^{-1} \end{pmatrix}.
\end{split}
\end{displaymath}
Thus $\DD^{S\otimes F}_+$ is block diagonal as a map with respect to (\ref{eq 3.2.30}) and
$$\DD^{S\otimes F}|_{K_F^+}=\wt{\alpha}_-^K\circ a_{E, \alpha}\circ(\wt{\alpha}_+^K)^{-1}:K_F^+\to K_F^-$$
is an isomorphism. Therefore $L_{F, \alpha}\to B$ satisfies the MF property for $\DD^{S\otimes F}$.

Denote by $P_{E, \alpha}:\pi_*E\to L_{E, \alpha}$ and $P_F:\pi_*F\to L_{F, \alpha}$ the $\Z_2$-graded projection maps with respect to (\ref{eq 3.2.29}) and (\ref{eq 3.2.30}), respectively. Note that the following diagram commutes.
\begin{center}
\begin{tikzcd}
\pi_*E \arrow{r}{\wt{\alpha}} \arrow{d}[swap]{P_{E, \alpha}} & \pi_*F \arrow{d}{P_F} \\ L_{E, \alpha} \arrow{r}[swap]{\wt{\alpha}^L} & L_{F, \alpha}
\end{tikzcd}
\end{center}
That is,
\begin{equation}\label{eq 3.2.32}
\wt{\alpha}^L\circ P_{E, \alpha}=P_F\circ\wt{\alpha}.
\end{equation}
Let $s\in\Gamma(B, L_{E, \alpha})$. By (\ref{eq 3.1.11}), (\ref{eq 3.2.28}) and (\ref{eq 3.2.32}),
\begin{displaymath}
\begin{split}
\nabla^{L_{E, \alpha}}s&=(P_{E, \alpha}\circ\nabla^{\pi_*E, u}\circ P_{E, \alpha})(s)\\
&=(P_{E, \alpha}\circ\nabla^{\pi_*E, u})(s)\\
&=(P_{E, \alpha}\circ\wt{\alpha}^{-1}\circ\nabla^{\pi_*F, u}\circ\wt{\alpha})(s)\\
&=((\wt{\alpha}^L)^{-1}\circ P_F\circ\nabla^{\pi_*F, u})(\wt{\alpha}^L(s))\\
&=((\wt{\alpha}^L)^{-1}\circ P_F\circ\nabla^{\pi_*F, u}\circ P_F)(\wt{\alpha}^L(s))\\
&=((\wt{\alpha}^L)^{-1}\circ\nabla^{L_{F, \alpha}})(\wt{\alpha}^L(s))\\
&=((\wt{\alpha}^L)^{-1}\circ\nabla^{L_{F, \alpha}}\circ\wt{\alpha}^L)(s).
\end{split}
\end{displaymath}
Thus
\begin{equation}\label{eq 3.2.33}
\nabla^{L_{E, \alpha}}=(\wt{\alpha}^L)^*\nabla^{L_{F, \alpha}}.
\end{equation}

Since $\wt{\alpha}:\pi_*E\to\pi_*F$ and $\wt{\alpha}^L:L_{E, \alpha}\to L_{F, \alpha}$ are $\Z_2$-graded isometric isomorphisms, the same is true for $\wh{\alpha}:=\wt{\alpha}\oplus\wt{\alpha}^{L, \op}:\pi_*E\oplus L_{E, \alpha}^{\op}\to\pi_*F\oplus L_{F, \alpha}^{\op}$. By the definition of $\wt{\DD}^{S\otimes E, \alpha}_+(z)$, it follows from (\ref{eq 3.2.26}) that
\begin{equation}\label{eq 3.2.34}
\wt{\DD}^{S\otimes E, \alpha}(z)=\wh{\alpha}^{-1}\circ\wt{\DD}^{S\otimes F}(z)\circ\wh{\alpha}
\end{equation}
for every $z\in\C$. On the other hand, by (\ref{eq 3.2.28}) and (\ref{eq 3.2.33}),
\begin{equation}\label{eq 3.2.35}
\nabla^{\pi_*E, u}\oplus\nabla^{L_{E, \alpha}, \op}=\wh{\alpha}^*(\nabla^{\pi_*F, u}\oplus\nabla^{L_{F, \alpha}, \op}).
\end{equation}
By the definition of $\wh{\bbb}^{E, \alpha}$, it follows from (\ref{eq 3.2.34}) and (\ref{eq 3.2.35}) that
$$\wh{\bbb}^{E, \alpha}=\wh{\alpha}^*\wh{\bbb}^F.$$
Since
\begin{displaymath}
\begin{split}
(\wh{\bbb}^{E, \alpha}_t)^{2k}&=(\wh{\alpha}^{-1}\circ\wh{\bbb}^F_t\circ\wh{\alpha})^{2k}=\wh{\alpha}^{-1}
\circ(\wh{\bbb}^F_t)^{2k}\circ\wh{\alpha},\\
\frac{d\wh{\bbb}^{E, \alpha}_t}{dt}&=\frac{d(\wh{\alpha}^*\wh{\bbb}^F_t)}{dt}=\wh{\alpha}^{-1}\circ\frac{d\wh{\bbb}^F_t}{dt}\circ\wh{\alpha},
\end{split}
\end{displaymath}
for any nonnegative integer $k$, it follows that
\begin{displaymath}
\begin{split}
\str\bigg(\frac{d\wh{\bbb}^{E, \alpha}_t}{dt}e^{-\frac{1}{2\pi i}(\wh{\bbb}^{E, \alpha}_t)^2}\bigg)&=\str\bigg(\wh{\alpha}^{-1}\circ
\frac{d\wh{\bbb}^F_t}{dt}e^{-\frac{1}{2\pi i}(\wh{\bbb}^F_t)^2}\circ\wh{\alpha}\bigg)\\
&=\str\bigg(\frac{d\wh{\bbb}^F_t}{dt}e^{-\frac{1}{2\pi i}(\wh{\bbb}^F_t)^2}\bigg).
\end{split}
\end{displaymath}
Thus (\ref{eq 3.2.24}) holds.
\end{proof}

We now prove Theorem \ref{thm 1.1}.
\begin{thm}\label{thm 3.1}
Let $\pi:X\to B$ be a submersion with closed, oriented and spin$^c$ fibers of even dimension, equipped with two sets of Riemannian and differential spin$^c$ structures
$$(T^H_0X, g^{T^VX}_0, g^\lambda_0, \nabla^\lambda_0),\qquad(T^H_1X, g^{T^VX}_1, g^\lambda_1, \nabla^\lambda_1),$$
where the underlying topological spin$^c$ structures coincide. Let $(E, g^E, \nabla^E)$ and $(F, g^F, \nabla^F)$ be triples. Denote by $\DD^{S\otimes E}$ and $\DD^{S\otimes F}$ the twisted spin$^c$ Dirac operators defined in terms of
$$(g^E, \nabla^E, T^H_0X, g^{T^VX}_0, g^\lambda_0, \nabla^\lambda_0)\qquad\textrm{ and }\qquad(g^F, \nabla^F, T^H_1X, g^{T^VX}_1, g^\lambda_1, \nabla^\lambda_1).$$
Let $(L_E, g^{L_E}, \nabla^{L_E})$ and $(L_F, g^{L_F}, \nabla^{L_F})$ be $\Z_2$-graded triples so that $L_E\to B$ and $L_F\to B$ satisfy the MF property for $\DD^{S\otimes E}$ and $\DD^{S\otimes F}$, respectively.

If there exists an isometric isomorphism $\alpha:(E, g^E)\to(F, g^F)$, then there exist balanced $\Z_2$-graded triples $(W_E, g^{W_E}, \nabla^{W_E})$ and $(W_F, g^{W_F}, \nabla^{W_F})$ and a $\Z_2$-graded isometric isomorphism
$$\wt{h}:(L_E\oplus W_E, g^{L_E}\oplus g^{W_E})\to(L_F\oplus W_F, g^{L_F}\oplus g^{W_F}),$$
which depends on $\alpha$, such that
\begin{equation}\label{eq 3.2.36}
\begin{split}
&\wh{\eta}^F(g^F, \nabla^F, T^H_1X, g^{T^VX}_1, g^\lambda_1, \nabla^\lambda_1, L_F)-\wh{\eta}^E(g^E, \nabla^E, T^H_0X, g^{T^VX}_0, g^\lambda_0, \nabla^\lambda_0, L_E)\\
&\equiv\int_{X/B}\big(T\wh{A}(\nabla^{T^VX}_0, \nabla^{T^VX}_1)\wedge e^{\frac{1}{2}c_1(\nabla^\lambda_0)}+\wh{A}(\nabla^{T^VX}_1)
\wedge e^{\frac{1}{2}Tc_1(\nabla^\lambda_0, \nabla^\lambda_1)}\big)\wedge\ch(\nabla^E)\\
&\quad+\int_{X/B}\todd(\nabla^{T^VX}_1)\wedge\CS(\nabla^E, \alpha^*\nabla^F)-\CS(\nabla^{L_E}\oplus\nabla^{W_E}, \wt{h}^*(\nabla^{L_F}\oplus\nabla^{W_F})).
\end{split}
\end{equation}
\end{thm}
\begin{proof}
Write $\DD^{S\otimes E, \alpha}$ for the twisted spin$^c$ Dirac operator defined in terms of
$$(g^E, \alpha^*\nabla^F, T^H_1X, g^{T^VX}_1, g^\lambda_1, \nabla^\lambda_1).$$
Let $(L_{E, \alpha}, g^{L_{E, \alpha}}, \nabla^{L_{E, \alpha}})$ be a $\Z_2$-graded triple so that $L_{E, \alpha}\to B$ satisfies the MF property for $\DD^{S\otimes E, \alpha}$. By Proposition \ref{prop 3.1}, there exist balanced $\Z_2$-graded triples $(W_0, g^{W_0}, \nabla^{W_0})$ and $(W_1, g^{W_1}, \nabla^{W_1})$ and a $\Z_2$-graded isometric isomorphism
\begin{equation}\label{eq 3.2.37}
h:(L_E\oplus W_0, g^{L_E}\oplus g^{W_0})\to(L_{E, \alpha}\oplus W_1, g^{L_{E, \alpha}}\oplus g^{W_1})
\end{equation}
such that
\begin{equation}\label{eq 3.2.38}
\begin{split}
&\wh{\eta}^E(g^E, \alpha^*\nabla^F, T^H_1X, g^{T^VX}_1, g^\lambda_1, \nabla^\lambda_1, L_{E, \alpha})-\wh{\eta}^E(g^E, \nabla^E, T^H_0X, g^{T^VX}_0, g^\lambda_0, \nabla^\lambda_0, L_E)\\
&\equiv\int_{X/B}\big(T\wh{A}(\nabla^{T^VX}_0, \nabla^{T^VX}_1)\wedge e^{\frac{1}{2}c_1(\nabla^\lambda_0)}+\wh{A}(\nabla^{T^VX}_1)\wedge
e^{\frac{1}{2}Tc_1(\nabla^\lambda_0, \nabla^\lambda_1)}\big)\wedge\ch(\nabla^E)\\
&\quad+\int_{X/B}\todd(\nabla^{T^VX}_1)\wedge\CS(\nabla^E, \alpha^*\nabla^F)-\CS(\nabla^{L_E}\oplus\nabla^{W_0}, h^*(\nabla^{L_{E, \alpha}}\oplus\nabla^{W_1})).
\end{split}
\end{equation}

By applying Proposition \ref{prop 3.2} to the isometric isomorphism $\alpha$, there exist a unique $\Z_2$-graded triple $(L_{F, \alpha}, g^{L_{F, \alpha}}, \nabla^{L_{F, \alpha}})$ and a unique $\Z_2$-graded isometric isomorphism
\begin{equation}\label{eq 3.2.39}
\wt{\alpha}^L:(L_{E, \alpha}, g^{L_{E, \alpha}})\to(L_{F, \alpha}, g^{L_{F, \alpha}})
\end{equation}
such that $L_{F, \alpha}\to B$ satisfies the MF property for $\DD^{S\otimes F}$ and
\begin{equation}\label{eq 3.2.40}
\wh{\eta}^E(g^E, \alpha^*\nabla^F, T^H_1X, g^{T^VX}_1, g^\lambda_1, \nabla^\lambda_1, L_{E, \alpha})=\wh{\eta}^F(g^F, \nabla^F, T^H_1X, g^{T^VX}_1, g^\lambda_1, \nabla^\lambda_1, L_{F, \alpha}).
\end{equation}

Since $(L_F, g^{L_F}, \nabla^{L_F})$ and $(L_{F, \alpha}, g^{L_{F, \alpha}}, \nabla^{L_{F, \alpha}})$ are $\Z_2$-graded triples so that both $L_F\to B$ and $L_{F, \alpha}\to B$ satisfy the MF property for $\DD^{S\otimes F}$, it follows from \cite[Corollary 1]{H20} that there exist balanced $\Z_2$-graded triples $(W_{F, 0}, g^{W_{F, 0}}, \nabla^{W_{F, 0}})$ and $(W_{F, 1}, g^{W_{F, 1}}, \nabla^{W_{F, 1}})$ and a $\Z_2$-graded isometric isomorphism
\begin{equation}\label{eq 3.2.41}
h_F:(L_{F, \alpha}\oplus W_{F, 0}, g^{L_{F, \alpha}}\oplus g^{W_{F, 0}})\to(L_F\oplus W_{F, 1}, g^{L_F}\oplus g^{W_{F, 1}})
\end{equation}
such that
\begin{equation}\label{eq 3.2.42}
\begin{split}
&\wh{\eta}^F(g^F, \nabla^F, T^H_1X, g^{T^VX}_1, g^\lambda_1, \nabla^\lambda_1, L_F)-\wh{\eta}^F(g^F, \nabla^F, T^H_1X, g^{T^VX}_1, g^\lambda_1, \nabla^\lambda_1, L_{F, \alpha})\\
&\equiv-\CS(\nabla^{L_{F, \alpha}}\oplus\nabla^{W_{F, 0}}, h_F^*(\nabla^{L_F}\oplus\nabla^{W_{F, 1}})).
\end{split}
\end{equation}
By (\ref{eq 3.2.37}), (\ref{eq 3.2.39}) and (\ref{eq 3.2.41}), the map
$$\wt{h}:(L_E\oplus W_0\oplus W_{F, 0}, g^{L_E}\oplus g^{W_0}\oplus g^{W_{F, 0}})\to(L_F\oplus W_1\oplus W_{F, 1}, g^{L_F}\oplus g^{W_1}\oplus g^{W_{F, 1}})$$
given by the composition
\begin{center}
\begin{tikzcd}
L_E\oplus W_0\oplus W_{F, 0} \arrow{d}{h\oplus\id_{W_{F, 0}}} \\ L_{E, \alpha}\oplus W_1\oplus W_{F, 0} \arrow{d}{\wt{\alpha}^L\oplus \id_{W_1}\oplus\id_{W_{F, 0}}} \\ L_{F, \alpha}\oplus W_1\oplus W_{F, 0} \arrow{d}{h_F\oplus\id_{W_1}} \\ L_F\oplus W_1\oplus W_{F, 1}
\end{tikzcd}
\end{center}
is a $\Z_2$-graded isometric isomorphism. On the other hand, (\ref{eq 3.2.40}) and (\ref{eq 3.2.42}) imply
\begin{equation}\label{eq 3.2.43}
\begin{split}
&\wh{\eta}^E(g^E, \alpha^*\nabla^F, T^H_1X, g^{T^VX}_1, g^\lambda_1, \nabla^\lambda_1, L_{E, \alpha})\\
&=\wh{\eta}^F(g^F, \nabla^F, T^H_1X, g^{T^VX}_1, g^\lambda_1, \nabla^\lambda_1, L_{F, \alpha})\\
&\equiv\wh{\eta}^F(g^F, \nabla^F, T^H_1X, g^{T^VX}_1, g^\lambda_1, \nabla^\lambda_1, L_F)+\CS(\nabla^{L_{F, \alpha}}\oplus\nabla^{W_{F, 0}}, h_F^*(\nabla^{L_F}\oplus\nabla^{W_{F, 1}})).
\end{split}
\end{equation}
By (\ref{eq 3.2.43}), (\ref{eq 3.2.38}) becomes
\begin{equation}\label{eq 3.2.44}
\begin{split}
&\wh{\eta}^F(g^F, \nabla^F, T^H_1X, g^{T^VX}_1, g^\lambda_1, \nabla^\lambda_1, L_F)-\wh{\eta}^E(g^E, \nabla^E, T^H_0X, g^{T^VX}_0, g^\lambda_0, \nabla^\lambda_0, L_E)\\
&\equiv\int_{X/B}\big(T\wh{A}(\nabla^{T^VX}_0, \nabla^{T^VX}_1)\wedge e^{\frac{1}{2}c_1(\nabla^\lambda_0)}+\wh{A}(\nabla^{T^VX}_1)\wedge
e^{\frac{1}{2}Tc_1(\nabla^\lambda_0, \nabla^\lambda_1)}\big)\wedge\ch(\nabla^E)\\
&\quad+\int_{X/B}\todd(\nabla^{T^VX}_1)\wedge\CS(\nabla^E, \alpha^*\nabla^F)-\CS(\nabla^{L_E}\oplus\nabla^{W_0}, h^*(\nabla^{L_{E, \alpha}}\oplus\nabla^{W_1}))\\
&\quad-\CS(\nabla^{L_{F, \alpha}}\oplus\nabla^{W_{F, 0}}, h_F^*(\nabla^{L_F}\oplus\nabla^{W_{F, 1}})).
\end{split}
\end{equation}
Since $\nabla^{L_{E, \alpha}}=(\wt{\alpha}^L)^*\nabla^{L_{F, \alpha}}$ by (\ref{eq 3.2.33}), it follows that
$$\nabla^{L_{E, \alpha}}\oplus\nabla^{W_1}\oplus\nabla^{W_{F, 0}}=(\wt{\alpha}^L\oplus\id_{W_1}\oplus\id_{W_{F, 0}})^*(\nabla^{L_{F, \alpha}}\oplus\nabla^{W_1}\oplus\nabla^{W_{F, 0}}).$$
Thus
\begin{equation}\label{eq 3.2.45}
\begin{split}
&\CS(\nabla^{L_{F, \alpha}}\oplus\nabla^{W_1}\oplus\nabla^{W_{F, 0}}, (h_F\oplus\id_{W_1})^*(\nabla^{L_F}\oplus\nabla^{W_1}\oplus\nabla^{W_{F, 1}}))\\
&\equiv\CS((h\oplus\id_{W_{F, 0}})^*(\wt{\alpha}^L\oplus\id_{W_1}\oplus\id_{W_{F, 0}})^*(\nabla^{L_{F, \alpha}}\oplus\nabla^{W_1}\oplus\nabla^{W_{F, 0}}), \\
&\qquad(h\oplus\id_{W_{F, 0}})^*(\wt{\alpha}^L\oplus\id_{W_1}\oplus\id_{W_{F, 0}})^*(h_F\oplus\id_{W_1})^*(\nabla^{L_F}\oplus\nabla^{W_1}\oplus\nabla^{W_{F, 1}}))\\
&\equiv\CS((h\oplus\id_{W_{F, 0}})^*(\nabla^{L_{E, \alpha}}\oplus\nabla^{W_1}\oplus\nabla^{W_{F, 0}}), \wt{h}^*(\nabla^{L_F}\oplus\nabla^{W_1}\oplus\nabla^{W_{F, 1}})).
\end{split}
\end{equation}
Since $\CS(\nabla^{W_{F, 0}}, \id_{W_{F, 0}}^*\nabla^{W_{F, 0}})\equiv 0$ and $\CS(\nabla^{W_1}, \id_{W_1}^*\nabla^{W_1})\equiv 0$, it follows from (\ref{eq 3.2.45}) and (\ref{eq 2.2.5}) that the sum of the last two terms of the right-hand side of (\ref{eq 3.2.44}) is equal to
\begin{equation}\label{eq 3.2.46}
\begin{split}
&\CS(\nabla^{L_E}\oplus\nabla^{W_0}, h^*(\nabla^{L_{E, \alpha}}\oplus\nabla^{W_1}))+\CS(\nabla^{L_{F, \alpha}}\oplus\nabla^{W_{F, 0}}, h_F^*(\nabla^{L_F}\oplus\nabla^{W_{F, 1}}))\\
&\equiv\CS(\nabla^{L_E}\oplus\nabla^{W_0}\oplus\nabla^{W_{F, 0}}, (h\oplus\id_{W_{F, 0}})^*(\nabla^{L_{E, \alpha}}\oplus\nabla^{W_1}
\oplus\nabla^{W_{F, 0}}))\\
&\quad+\CS(\nabla^{L_{F, \alpha}}\oplus\nabla^{W_1}\oplus\nabla^{W_{F, 0}}, (h_F\oplus\id_{W_1})^*(\nabla^{L_F}\oplus\nabla^{W_1}\oplus
\nabla^{W_{F, 1}}))\\
&\equiv\CS(\nabla^{L_E}\oplus\nabla^{W_0}\oplus\nabla^{W_{F, 0}}, (h\oplus\id_{W_{F, 0}})^*(\nabla^{L_{E, \alpha}}\oplus\nabla^{W_1}
\oplus\nabla^{W_{F, 0}}))\\
&\quad+\CS((h\oplus\id_{W_{F, 0}})^*(\nabla^{L_{E, \alpha}}\oplus\nabla^{W_1}\oplus\nabla^{W_{F, 0}}), \wt{h}^*(\nabla^{L_F}\oplus\nabla^{W_1}\oplus\nabla^{W_{F, 1}}))\\
&\equiv\CS(\nabla^{L_E}\oplus\nabla^{W_0}\oplus\nabla^{W_{F, 0}}, \wt{h}^*(\nabla^{L_F}\oplus\nabla^{W_1}\oplus\nabla^{W_{F, 1}})).
\end{split}
\end{equation}
By taking $W_E=W_0\oplus W_{F, 0}$ and $W_F=W_1\oplus W_{F, 1}$ and similarly for $g^{W_E}$, $g^{W_F}$ and $\nabla^{W_E}$, $\nabla^{W_F}$, (\ref{eq 3.2.44}) and (\ref{eq 3.2.46}) show that (\ref{eq 3.2.36}) holds.
\end{proof}

Note that Remark \ref{remark 3.1} applies to Theorem \ref{thm 3.1} as well.

\section{Applications of the extended variational formula for the Bismut--Cheeger eta form}\label{s 4}

In this section we present some applications of Proposition \ref{prop 1.1} and Theorem \ref{thm 1.1}. All the results in this subsection are under the following setup. Let $\pi:X\to B$ be a submersion with closed, oriented and spin$^c$ fibers of even dimension, equipped with a fixed Riemannian and differential spin$^c$ structure $(T^HX, g^{T^VX}, g^\lambda, \nabla^\lambda)$. In this section, we suppress the dependence of the Bismut--Cheeger eta form on $(T^HX, g^{T^VX}, g^\lambda, \nabla^\lambda)$ when no confusion arises.

\subsection{$\Z_2$-graded additivity of the Bismut--Cheeger eta form}\label{s 4.1}

In this subsection, we establish some intermediate results for the Bismut--Cheeger eta form. These results, along with Proposition \ref{prop 3.2}, are used to prove the $\Z_2$-graded additivity of the Bismut--Cheeger eta form (Theorem \ref{thm 1.2}).

The following proposition says the Bismut--Cheeger eta form is additive with respect to direct sums.
\begin{prop}\label{prop 4.1}
Let $(E, g^E, \nabla^E)$ and $(F, g^F, \nabla^F)$ be two $\Z_2$-graded triples over $X$. If $(L_E, g^{L_E}, \nabla^{L_E})$ and $(L_F, g^{L_F}, \nabla^{L_F})$ are $\Z_2$-graded triples so that $L_E\to B$ and $L_F\to B$ satisfy the MF property for $\DD^{S\otimes E}$ and $\DD^{S\otimes F}$, respectively, then there exists a $\Z_2$-graded triple
$$(L_{E\oplus F}, g^{L_{E\oplus F}}, \nabla^{L_{E\oplus F}})$$
such that $L_{E\oplus F}\to B$ satisfies the MF property for $\DD^{S\otimes(E\oplus F)}$ and
$$\wh{\eta}^{E\oplus F}(g^E\oplus g^F, \nabla^E\oplus\nabla^F, L_{E\oplus F})\equiv\wh{\eta}^E(g^E, \nabla^E, L_E)+\wh{\eta}^F(g^F, \nabla^F, L_F).$$
\end{prop}
\begin{proof}
Let $K_E\to B$ and $K_F\to B$ be $\Z_2$-graded closed subbundles of $\pi_*E\to B$ and $\pi_*F\to B$ that are complementary to $L_E\to B$ and $L_F\to B$, respectively, i.e.
\begin{equation}\label{eq 4.1.1}
\begin{split}
(\pi_*E)^+=K_E^+\oplus L_E^+,\qquad(\pi_*E)^-=K_E^-\oplus L_E^-,\\
(\pi_*F)^+=K_F^+\oplus L_F^+,\qquad(\pi_*F)^-=K_F^-\oplus L_F^-,
\end{split}
\end{equation}
$\DD^{S\otimes E}_+$ and $\DD^{S\otimes F}_+$ are block diagonal as maps with respect to (\ref{eq 4.1.1}), say
\begin{equation}\label{eq 4.1.2}
\DD^{S\otimes E}_+=\begin{pmatrix} a_E & 0 \\ 0 & d_E \end{pmatrix},\qquad\DD^{S\otimes F}_+=\begin{pmatrix} a_F & 0 \\ 0 & d_F \end{pmatrix},
\end{equation}
and $a_E:K_E^+\to K_E^-$ and $a_F:K_F^+\to K_F^-$ are isomorphisms.

Define a $\Z_2$-graded triple $(L_{E\oplus F}, g^{L_{E\oplus F}}, \nabla^{L_{E\oplus F}})$ by
\begin{equation}\label{eq 4.1.3}
L_{E\oplus F}=L_E\oplus L_F,\qquad g^{L_{E\oplus F}}=g^{L_E}\oplus g^{L_F},\qquad\nabla^{L_{E\oplus F}}=\nabla^{L_E}\oplus\nabla^{L_F}.
\end{equation}
By (\ref{eq 4.1.1}),
\begin{equation}\label{eq 4.1.4}
\begin{split}
(\pi_*(E\oplus F))^+&=(\pi_*E)^+\oplus(\pi_*F)^+=(K_E^+\oplus K_F^+)\oplus(L_E^+\oplus L_F^+),\\
(\pi_*(E\oplus F))^-&=(\pi_*E)^-\oplus(\pi_*F)^-=(K_E^-\oplus K_F^-)\oplus(L_E^-\oplus L_F^-).
\end{split}
\end{equation}
Thus $K_E\oplus K_F\to B$ is complementary to $L_{E\oplus F}\to B$. With respect to (\ref{eq 4.1.4}), $\DD^{S\otimes(E\oplus F)}_+$ is given by
$$\DD^{S\otimes(E\oplus F)}_+=\begin{pmatrix} a_E & 0 & 0 & 0 \\ 0 & a_F & 0 & 0 \\ 0 & 0 & d_E & 0 \\ 0 & 0 & 0 & d_F \end{pmatrix},$$
which is block diagonal as a map with respect to (\ref{eq 4.1.4}), and shows that $\DD^{S\otimes(E\oplus F)}_+|_{K_E^+\oplus K_F^+}=a_E\oplus a_F$ is an isomorphism. Thus $L_{E\oplus F}\to B$ satisfies the MF property for $\DD^{S\otimes(E\oplus F)}$.

We show that the Bismut superconnection $\wh{\bbb}^{E\oplus F}$ is additive in the following sense:
\begin{equation}\label{eq 4.1.5}
\wh{\bbb}^{E\oplus F}=\wh{\bbb}^E\oplus\wh{\bbb}^F.
\end{equation}
With respect to the decomposition
$$\pi_*(E\oplus F)\oplus L_{E\oplus F}^{\op}=(\pi_*E\oplus L_E^{\op})\oplus(\pi_*F\oplus L_F^{\op}),$$
we have
$$\DD^{S\otimes(E\oplus F)}_+=\DD^{S\otimes E}_+\oplus\DD^{S\otimes F}_+,\quad i_{E\oplus F}^-=i_E^-\oplus i_F^-,\quad P_{E\oplus F}^+=P_E^+\oplus P_F^+.$$
Thus $\wt{\DD}^{S\otimes(E\oplus F)}_+(z)=\wt{\DD}^{S\otimes E}_+(z)\oplus\wt{\DD}^{S\otimes F}_+(z)$ for any $z\in\C$, and therefore
\begin{equation}\label{eq 4.1.6}
\wt{\DD}^{S\otimes(E\oplus F)}(z)=\wt{\DD}^{S\otimes E}(z)\oplus\wt{\DD}^{S\otimes F}(z).
\end{equation}
Since $\nabla^{\pi_*(E\oplus F), u}=\nabla^{\pi_*E, u}\oplus\nabla^{\pi_*F, u}$, it follows that
\begin{equation}\label{eq 4.1.7}
\nabla^{\pi_*(E\oplus F), u}\oplus\nabla^{L_{E\oplus F}, \op}=(\nabla^{\pi_*E, u}\oplus\nabla^{L_E, \op})\oplus(\nabla^{\pi_*F, u}\oplus\nabla^{L_F, \op}).
\end{equation}
By (\ref{eq 4.1.6}) and (\ref{eq 4.1.7}), (\ref{eq 4.1.5}) holds.

Let $t, T\in(0, \infty)$ satisfy $t<T$. By (\ref{eq 4.1.5}) and (\ref{eq 2.2.6}),
\begin{equation}\label{eq 4.1.8}
\CS(\wh{\bbb}^{E\oplus F}_t, \wh{\bbb}^{E\oplus F}_T)\equiv\CS(\wh{\bbb}^E_t\oplus\wh{\bbb}^F_t, \wh{\bbb}^E_T\oplus\wh{\bbb}^F_T)\equiv\CS(\wh{\bbb}^E_t, \wh{\bbb}^E_T)+\CS(\wh{\bbb}^F_t, \wh{\bbb}^F_T).
\end{equation}
By letting $t\to 0$ and $T\to\infty$ in (\ref{eq 4.1.8}), the result follows.
\end{proof}

Proposition \ref{prop 4.1} is also valid if not all of the triples $(E, g^E, \nabla^E)$ and $(F, g^F, \nabla^F)$ are $\Z_2$-graded.

The following lemma says for any given $\Z_2$-graded triple $(E, g^E, \nabla^E)$, the $\Z_2$-graded complex vector bundle satisfying the MF property for $\DD^{S\otimes E}$ can be expressed in terms of those for $\DD^{S\otimes E^+}$ and $\DD^{S\otimes E^-}$.
\begin{lemma}\label{lemma 4.1}
Let $(E, g^E, \nabla^E)$ be a $\Z_2$-graded triple over $X$. If $L_{E^+}\to B$ and $L_{E^-}\to B$ are $\Z_2$-graded complex vector bundles satisfying the MF property for $\DD^{S\otimes E^+}$ and $\DD^{S\otimes E^-}$, respectively, then $L_{E^+}\oplus L_{E^-}^{\op}\to B$ is a $\Z_2$-graded complex vector bundle satisfying the MF property for $\DD^{S\otimes E}$.
\end{lemma}
\begin{proof}
Let $K_{E^+}\to B$ and $K_{E^-}\to B$ be $\Z_2$-graded closed subbundles of $\pi_*E^+\to B$ and $\pi_*E^-\to B$ that are complementary to $L_{E^+}\to B$ and $L_{E^-}\to B$, respectively, i.e.
\begin{equation}\label{eq 4.1.9}
\begin{split}
(\pi_*E^+)^+&=K_{E^+}^+\oplus L_{E^+}^+,\qquad(\pi_*E^+)^-=K_{E^+}^-\oplus L_{E^+}^-,\\
(\pi_*E^-)^+&=K_{E^-}^+\oplus L_{E^-}^+,\qquad(\pi_*E^-)^-=K_{E^-}^-\oplus L_{E^-}^-,
\end{split}
\end{equation}
$\DD^{S\otimes E^+}_+$ and $\DD^{S\otimes E^-}_+$ are block diagonal as maps with respect to (\ref{eq 4.1.9}), say
$$\DD^{S\otimes E^+}_+=\begin{pmatrix} a_{E^+} & 0 \\ 0 & d_{E^+} \end{pmatrix},\qquad\DD^{S\otimes E^-}_+=\begin{pmatrix} a_{E^-} & 0 \\ 0 & d_{E^-} \end{pmatrix},$$
and $a_{E^+}:K_{E^+}^+\to K_{E^+}^-$ and $a_{E^-}:K_{E^-}^+\to K_{E^-}^-$ are isomorphisms.

By (\ref{eq 4.1.9}),
\begin{equation}\label{eq 4.1.10}
\begin{split}
(\pi_*E)^+&=(\pi_*E^+)^+\oplus(\pi_*E^-)^-\\
&=K_{E^+}^+\oplus L_{E^+}^+\oplus  K_{E^-}^-\oplus L_{E^-}^-\\
&=(K_{E^+}\oplus K_{E^-}^{\op})^+\oplus(L_{E^+}\oplus L_{E^-}^{\op})^+
\end{split}
\end{equation}
and
\begin{equation}\label{eq 4.1.11}
\begin{split}
(\pi_*E)^-&=(\pi_*E^+)^-\oplus(\pi_*E^-)^+\\
&=K_{E^+}^-\oplus L_{E^+}^-\oplus K_{E^-}^+\oplus L_{E^-}^+\\
&=(K_{E^+}\oplus K_{E^-}^{\op})^-\oplus(L_{E^+}\oplus L_{E^-}^{\op})^-.
\end{split}
\end{equation}
By (\ref{eq 4.1.10}) and (\ref{eq 4.1.11}), $\pi_*E=(K_{E^+}\oplus K_{E^-}^{\op})\oplus(L_{E^+}\oplus L_{E^-}^{\op})$. With respect to (\ref{eq 4.1.10}) and (\ref{eq 4.1.11}), $\DD^{S\otimes E}_+$ is given by
$$\DD^{S\otimes E}_+=\begin{pmatrix} a_{E^+} & 0 & 0 & 0 \\ 0 & a_{E^-}^* & 0 & 0 \\ 0 & 0 & d_{E^+} & 0 \\ 0 & 0 & 0 & d_{E^-}^* \end{pmatrix},$$
where $\DD^{S\otimes E^-}_-=(\DD^{S\otimes E^-}_+)^*=\begin{pmatrix} a_{E^-}^* & 0 \\ 0 & d_{E^-}^* \end{pmatrix}$. Note that $\DD^{S\otimes E}_+|_{(K_{E^+}\oplus K_{E^-}^{\op})^+}=a_{E^+}\oplus a_{E^-}^*$, which is an isomorphism. Thus $L_{E^+}\oplus L_{E^-}^{\op}\to B$ satisfies the MF property for $\DD^{S\otimes E}$.
\end{proof}

Let $E\to X$ be a $\Z_2$-graded complex vector bundle. Since $\End(E)^+=\End(E^+)\oplus\End(E^-)$, it follows that
$$\End(E^{\op})^+=\End(E^{\op, +})\oplus\End(E^{\op, -})=\End(E^-)\oplus\End(E^+)=\End(E)^+.$$
Thus $T\in\End(E^{\op})^+$ if and only if $T\in\End(E)^+$. Moreover, $\str_{E^{\op}}(T)$ exists if and only if $\str_E(T)$ exists. If either one of them exists, then
\begin{equation}\label{eq 4.1.12}
\str_{E^{\op}}(T)=\tr(T|_{E^{\op, +}})-\tr(T|_{E^{\op, -}})=\tr(T|_{E^-})-\tr(T|_{E^+})=-\str_E(T).
\end{equation}

The following lemma says switching the $\Z_2$-grading of a $\Z_2$-graded triple induces a minus sign in the corresponding Bismut--Cheeger eta form.
\begin{lemma}\label{lemma 4.2}
Let $(E, g^E, \nabla^E)$ be a $\Z_2$-graded triple over $X$. If $(L, g^L, \nabla^L)$ is a $\Z_2$-graded triple so that $L\to B$ satisfies the MF property for $\DD^{S\otimes E}$, then $(L^{\op}, g^{L, \op}, \nabla^{L, \op})$ is a $\Z_2$-graded triple so that $L^{\op}\to B$ satisfies the MF property for $\DD^{S\otimes E^{\op}}$, and
\begin{equation}\label{eq 4.1.13}
\wh{\eta}^{E^{\op}}(g^{E, \op}, \nabla^{E, \op}, L^{\op})\equiv-\wh{\eta}^E(g^E, \nabla^E, L).
\end{equation}
\end{lemma}
\begin{proof}
Let $K\to B$ be a $\Z_2$-graded closed subbundle of $\pi_*E\to B$ that is complementary to $L\to B$, i.e.
\begin{equation}\label{eq 4.1.14}
(\pi_*E)^+=K^+\oplus L^+,\qquad(\pi_*E)^-=K^-\oplus L^-,
\end{equation}
$\DD^{S\otimes E}_+$ is block diagonal as a map with respect to (\ref{eq 4.1.14}) and $\DD^{S\otimes E}|_{K^+}$ is an isomorphism. Since
\begin{displaymath}
\begin{split}
(S(T^VX)\otimes E^{\op})^+&=(S(T^VX)^+\otimes E^-)\oplus(S(T^VX)^-\otimes E^+)=(S(T^VX)\otimes E)^-,\\
(S(T^VX)\otimes E^{\op})^-&=(S(T^VX)^+\otimes E^+)\oplus(S(T^VX)^-\otimes E^-)=(S(T^VX)\otimes E)^+,
\end{split}
\end{displaymath}
it follows that
\begin{equation}\label{eq 4.1.15}
\begin{split}
(\pi_*(E^{\op}))^+&=(\pi_*E)^-=K^-\oplus L^-,\\
(\pi_*(E^{\op}))^-&=(\pi_*E)^+=K^+\oplus L^+.
\end{split}
\end{equation}
That is, $\pi_*(E^{\op})=K^{\op}\oplus L^{\op}$. Moreover, $\DD^{S\otimes E^{\op}}_+=\DD^{S\otimes E}_-=(\DD^{S\otimes E}_+)^*$. Thus $\DD^{S\otimes E^{\op}}_+$ is also block diagonal as a map with respect to (\ref{eq 4.1.15}), and $\DD^{S\otimes E^{\op}}_+|_{K^{\op, +}}$ is the inverse of $\DD^{S\otimes E}_+|_{K^+}$, which is an isomorphism. Thus $L^{\op}\to B$ satisfies the MF property for $\DD^{S\otimes E^{\op}}$.

Note that $g^{L, \op}$ is inherited from $g^{\pi_*(E^{\op})}$ and $\nabla^{L, \op}=P^{\op}\circ\nabla^{\pi_*(E^{\op}), u}\circ P^{\op}$, where $P^{\op}:\pi_*(E^{\op})\to L^{\op}$ is the obvious $\Z_2$-graded projection map. By applying (\ref{eq 4.1.12}) to \dis{\frac{d\wh{\bbb}^{E^{\op}}_t}{dt}e^{-\frac{1}{2\pi i}(\wh{\bbb}^{E^{\op}}_t)^2}}, (\ref{eq 4.1.13}) holds.
\end{proof}
\begin{prop}\label{prop 4.2}
Let $(V, g^V, \nabla^V)$ be a balanced $\Z_2$-graded triple over $X$. If $(L_+, g^{L_+}, \nabla^{L_+})$ is a $\Z_2$-graded triple so that $L_+\to B$ satisfies the MF property for $\DD^{S\otimes V^+}$, then there exists a balanced $\Z_2$-graded triple $(L, g^L, \nabla^L)$ such that $L\to B$ satisfies the MF property for $\DD^{S\otimes V}$ and
$$\wh{\eta}^V(g^V, \nabla^V, L)\equiv 0.$$
\end{prop}
\begin{proof}
Since the $\Z_2$-graded triple $(V, g^V, \nabla^V)$ is balanced, $(V^+, g^{V, +}, \nabla^{V, +})=(V^-, g^{V, -}, \nabla^{V, -})$. Thus $L_+\to B$ also satisfies the MF property for $\DD^{S\otimes V^-}$. Define a $\Z_2$-graded triple $(L, g^L, \nabla^L)$ by
$$L=L_+\oplus L_+^{\op},\qquad g^L=g^{L_+}\oplus g^{L_+, \op},\qquad\nabla^L=\nabla^{L_+}\oplus\nabla^{L_+, \op}.$$
By Lemma \ref{lemma 4.1}, $L\to B$ satisfies the MF property for $\DD^{S\otimes V}$. Since
$$L^+=L_+^+\oplus L_+^-\qquad\textrm{ and }\qquad L^-=L_+^-\oplus L_+^+,$$
and similarly for $g^L$ and $\nabla^L$, the $\Z_2$-graded triple $(L, g^L, \nabla^L)$ is balanced. Since $(V^{\op}, g^{V, \op}, \nabla^{V, \op})=(V, g^V, \nabla^V)$, it follows from Lemma \ref{lemma 4.2} and the fact $(L^{\op}, g^{L, \op}, \nabla^{L, \op})=(L, g^L, \nabla^L)$ that
$$\wh{\eta}^V(g^V, \nabla^V, L)\equiv-\wh{\eta}^{V^{\op}}(g^{V, \op}, \nabla^{V, \op}, L^{\op})\equiv-\wh{\eta}^V(g^V, \nabla^V, L).$$
Thus $\wh{\eta}^V(g^V, \nabla^V, L)\equiv 0$.
\end{proof}

We now prove Theorem \ref{thm 1.2}.
\begin{thm}\label{thm 4.1}
Let $(E, g^E, \nabla^E)$ be a $\Z_2$-graded triple over $X$. If $(L_{E^+}, g^{L_{E^+}}, \nabla^{L_{E_+}})$ and $(L_{E^-}, g^{L_{E^-}}, \nabla^{L_{E^-}})$ are $\Z_2$-graded triples so that $L_{E^+}\to B$ and $L_{E^-}\to B$ satisfy the MF property for $\DD^{S\otimes E^+}$ and $\DD^{S\otimes E^-}$, respectively, then
\begin{equation}\label{eq 4.1.16}
\wh{\eta}^E(g^E, \nabla^E, L_{E^+}\oplus L_{E^-}^{\op})\equiv\wh{\eta}^{E^+}(g^{E, +}, \nabla^{E, +}, L_{E^+})-\wh{\eta}^{E^-}(g^{E, -}, \nabla^{E, -}, L_{E^-}).
\end{equation}
\end{thm}
\begin{proof}
By Lemma \ref{lemma 4.1}, $L_{E^+}\oplus L_{E^-}^{\op}\to B$ satisfies the MF property for $\DD^{S\otimes E}$. Thus the left-hand side of (\ref{eq 4.1.16}) makes sense.

Consider the (ungraded) triple $(E\oplus E^-, g^E\oplus g^{E, -}, \nabla^E\oplus\nabla^{E, -})$. By Proposition \ref{prop 4.1},
$$\big((L_{E^+}\oplus L_{E^-}^{\op})\oplus L_{E^-}, (g^{L_{E^+}}\oplus g^{L_{E^-}, \op})\oplus g^{L_{E^-}}, (\nabla^{L_{E^+}}\oplus\nabla^{L_{E^-}, \op})\oplus\nabla^{L_{E^-}}\big)$$
is a $\Z_2$-graded triple so that $(L_{E^+}\oplus L_{E^-}^{\op})\oplus L_{E^-}\to B$ satisfies the MF property for $\DD^{S\otimes(E\oplus E^-)}$, which is defined in terms of $\nabla^E\oplus\nabla^{E, -}$.

On the other hand, define a balanced $\Z_2$-graded triple $(H, g^H, \nabla^H)$ by $H^\pm:=E^-$, and similarly for $g^H$ and $\nabla^H$. Consider the (ungraded) triple $(E^+\oplus H, g^{E, +}\oplus g^H, \nabla^{E, +}\oplus\nabla^H)$. The map
$$\alpha:(E\oplus E^-, g^E\oplus g^{E, -})\to (E^+\oplus H, g^{E, +}\oplus g^H)$$
defined by
$$\alpha((a, b), c)=(a, (c, b))$$
is obviously an isometric isomorphism and satisfies
\begin{equation}\label{eq 4.1.17}
\alpha^*(\nabla^{E, +}\oplus\nabla^H)=\nabla^E\oplus\nabla^{E, -}.
\end{equation}
By applying Proposition \ref{prop 3.2} to the isometric isomorphism $\alpha$ and noting that $(L_{E^+}\oplus L_{E^-}^{\op})\oplus L_{E^-}\to B$ satisfies the MF property for $\DD^{S\otimes(E\oplus E^-)}$, which is defined in terms of $\nabla^E\oplus\nabla^{E, -}=\alpha^*(\nabla^{E, +}\oplus\nabla^H)$ by (\ref{eq 4.1.17}), there exists a unique $\Z_2$-graded triple $(L_\alpha, g^{L_\alpha}, \nabla^{L_\alpha})$ and a unique $\Z_2$-graded isometric isomorphism
$$\wt{\alpha}^L:\big((L_{E^+}\oplus L_{E^-}^{\op})\oplus L_{E^-}, (g^{L_{E^+}}\oplus g^{L_{E^-}, \op})\oplus g^{L_{E^-}}\big)\to(L_\alpha, g^{L_\alpha})$$
such that $L_\alpha\to B$ satisfies the MF property for $\DD^{S\otimes(E^+\oplus H)}$, which is defined in terms of $\nabla^{E, +}\oplus\nabla^H$,  and
\begin{equation}\label{eq 4.1.18}
\begin{split}
&\wh{\eta}^{E\oplus E^-}(g^E\oplus g^{E, -}, \alpha^*(\nabla^{E, +}\oplus\nabla^H), (L_{E^+}\oplus L_{E^-}^{\op})\oplus L_{E^-})\\
&=\wh{\eta}^{E^+\oplus H}(g^{E, +}\oplus g^H, \nabla^{E, +}\oplus\nabla^H, L_\alpha).
\end{split}
\end{equation}
By the definition of $\alpha$, the $\Z_2$-graded triple $(L_\alpha, g^{L_\alpha}, \nabla^{L_\alpha})$ is given by
\begin{equation}\label{eq 4.1.19}
\big(L_{E^+}\oplus(L_{E^-}\oplus L_{E^-}^{\op}), g^{L_{E^+}}\oplus(g^{L_{E^-}}\oplus g^{L_{E^-}, \op}), \nabla^{L_{E^+}}\oplus(\nabla^{L_{E^-}}\oplus\nabla^{L_{E^-}, \op})\big).
\end{equation}
By (\ref{eq 4.1.17}) and (\ref{eq 4.1.19}), (\ref{eq 4.1.18}) becomes
\begin{equation}\label{eq 4.1.20}
\begin{split}
&\wh{\eta}^{E\oplus E^-}(g^E\oplus g^{E, -}, \nabla^E\oplus\nabla^{E, -}, (L_{E^+}\oplus L_{E^-}^{\op})\oplus L_{E^-})\\
&=\wh{\eta}^{E^+\oplus H}(g^{E, +}\oplus g^H, \nabla^{E, +}\oplus\nabla^H, L_{E^+}\oplus(L_{E^-}\oplus L_{E^-}^{\op})).
\end{split}
\end{equation}
By Propositions \ref{prop 4.1} and \ref{prop 4.2}, (\ref{eq 4.1.20}) becomes
\begin{displaymath}
\begin{split}
&\wh{\eta}^E(g^E, \nabla^E, L_{E^+}\oplus L_{E^-}^{\op})+\wh{\eta}^{E^-}(g^{E, -}, \nabla^{E, -}, L_{E^-})\\
&\equiv\wh{\eta}^{E\oplus E^-}(g^E\oplus g^{E, -}, \nabla^E\oplus\nabla^{E, -}, (L_{E^+}\oplus L_{E^-}^{\op})\oplus L_{E^-})\\
&=\wh{\eta}^{E^+\oplus H}(g^{E, +}\oplus g^H, \nabla^{E, +}\oplus\nabla^H, L_{E^+}\oplus(L_{E^-}\oplus L_{E^-}^{\op}))\\
&\equiv\wh{\eta}^{E^+}(g^{E, +}, \nabla^{E, +}, L_{E^+})+\wh{\eta}^H(g^H, \nabla^H, L_{E^-}\oplus L_{E^-}^{\op})\\
&\equiv\wh{\eta}^{E^+}(g^{E, +}, \nabla^{E, +}, L_{E^+}).
\end{split}
\end{displaymath}
Thus (\ref{eq 4.1.16}) holds.
\end{proof}

\subsection{The analytic index in differential and $\R/\Z$ $K$-theory}\label{s 4.2}

In this subsection we prove that the analytic index in differential $K$-theory $\ind^a_{\wh{K}}:\wh{K}_{\FL}(X)\to\wh{K}_{\FL}(B)$ is a well defined group homomorphism (Proposition \ref{prop 1.2}).

\begin{remark}
Here is a remark on representing elements in the topological $K$-group $K(X)$ using $\Z_2$-graded complex vector bundles. First note that every $\Z_2$-graded complex vector bundle $E\to X$ represents the element $[E^+]-[E^-]$ in $K(X)$. Conversely, given an element in $K(X)$ written as a formal difference $[E]-[F]$, by defining a $\Z_2$-graded complex vector bundle $H\to X$ by $H^+:=E$ and $H^-:=F$, we see that $H\to X$ represents $[E]-[F]$.

By \cite[p.289]{BGV}, two $\Z_2$-graded complex vector bundles $E\to X$ and $F\to X$ represent the same element in $K(X)$ if there exist complex vector bundles $G\to X$ and $H\to X$ such that
$$E^+\oplus G\cong F^+\oplus H,\qquad E^-\oplus G\cong F^-\oplus H.$$
In other words, $E\to X$ and $F\to X$ represent the same element in $K(X)$ if there exist $\Z_2$-graded complex vector bundles $\wh{G}\to X$ and $\wh{H}\to X$ of the form $\wh{G}^+=\wh{G}^-$ and $\wh{H}^+=\wh{H}^-$ such that $E\oplus\wh{G}\cong F\oplus\wh{H}$ as $\Z_2$-graded complex vector bundles.
\end{remark}

The Freed--Lott differential $K$-group $\wh{K}_{\FL}(X)$ \cite[Definition 2.16]{FL10} can be described in terms of $\Z_2$-graded generators of the form $\E=(E, g^E, \nabla^E, \omega)$, where $(E, g^E, \nabla^E)$ is a $\Z_2$-graded triple and \dis{\omega\in\frac{\Omega^{\odd}(X)}{\im(d)}}. Two $\Z_2$-graded generators $\E$ and $\F$ are equal in $\wh{K}_{\FL}(X)$ if and if only there exist balanced $\Z_2$-graded triples $(V_E, g^{V_E}, \nabla^{V_E})$ and $(V_F, g^{V_F}, \nabla^{V_F})$ and a $\Z_2$-graded isometric isomorphism
$\alpha:(E\oplus V_E, g^E\oplus g^{V_E})\to(F\oplus V_F, g^F\oplus g^{V_F})$ such that
\begin{equation}\label{eq 4.2.1}
\omega_E-\omega_F\equiv\CS(\nabla^E\oplus\nabla^{V_E}, \alpha^*(\nabla^F\oplus\nabla^{V_F})).
\end{equation}

The analytic index of a $\Z_2$-graded generator $\E$ of $\wh{K}_{\FL}(X)$ is defined to be
\begin{equation}\label{eq 4.2.2}
\ind^a_{\wh{K}}(\E; L)=\bigg(L, g^L, \nabla^L, \int_{X/B}\todd(\nabla^{T^VX})\wedge\omega+\wh{\eta}^E(g^E, \nabla^E, T^HX, g^{T^VX}, g^\lambda, \nabla^\lambda, L)\bigg),
\end{equation}
where $(L, g^L, \nabla^L)$ is a $\Z_2$-graded triple so that $L\to B$ satisfies the MF property for $\DD^{S\otimes E}$. By \cite[Corollary 1]{H20}, $\ind^a_{\wh{K}}(\E; L)$ does not depend on the choice of $(L, g^L, \nabla^L)$ so that $L\to B$ satisfies the MF property for $\DD^{S\otimes E}$. Henceforth we write $\ind^a_{\wh{K}}(\E)$ for $\ind^a_{\wh{K}}(\E; L)$.

Let $\E$ and $\F$ be $\Z_2$-graded generators of $\wh{K}_{\FL}(X)$. By \cite[Corollary 1]{H20} and Proposition \ref{prop 4.1}, the analytic index of $\ind^a_{\wh{K}}(\E+\F)$ is given by
\begin{equation}\label{eq 4.2.3}
\begin{split}
\ind^a_{\wh{K}}(\E+\F)&=\bigg(L_{E\oplus F}, g^{L_{E\oplus F}}, \nabla^{L_{E\oplus F}}, \int_{X/B}\todd(\nabla^{T^VX})\wedge(\omega_E+\omega_F),\\ &\qquad\wh{\eta}^{E\oplus F}(g^E\oplus g^F, \nabla^E\oplus\nabla^F, T^HX, g^{T^VX}, g^\lambda, \nabla^\lambda, L_{E\oplus F})\bigg),
\end{split}
\end{equation}
where $(L_{E\oplus F}, g^{L_{E\oplus F}}, \nabla^{L_{E\oplus F}})$ is the $\Z_2$-graded triple given by (\ref{eq 4.1.3}). Thus the analytic index in differential $K$-theory is additive, i.e.
\begin{equation}\label{eq 4.2.4}
\ind^a_{\wh{K}}(\E+\F)=\ind^a_{\wh{K}}(\E)+\ind^a_{\wh{K}}(\F).
\end{equation}

We now prove Proposition \ref{prop 1.2}.
\begin{proof}[Proof of Proposition \ref{prop 1.2}]
We first show that $\ind^a_{\wh{K}}$ passes to a well defined map $\wh{K}_{\FL}(X)\to\wh{K}_{\FL}(B)$, i.e. if $\E$ and $\F$ are $\Z_2$-graded generators of $\wh{K}_{\FL}(X)$ satisfying $\E=\F$, then
\begin{equation}\label{eq 4.2.5}
\ind^a_{\wh{K}}(\E)=\ind^a_{\wh{K}}(\F).
\end{equation}
Once (\ref{eq 4.2.5}) is established, (\ref{eq 4.2.4}) immediately implies that $\ind^a_{\wh{K}}:\wh{K}_{\FL}(X)\to\wh{K}_{\FL}(B)$ is a group homomorphism.

Since $\E=\F$, there exist balanced $\Z_2$-graded triples $(V_E, g^{V_E}, \nabla^{V_E})$ and $(V_F, g^{V_F}, \nabla^{V_F})$ and a $\Z_2$-graded isometric isomorphism
$$\alpha:(E\oplus V_E, g^E\oplus g^{V_E})\to(F\oplus V_F, g^F\oplus g^{V_E})$$
such that (\ref{eq 4.2.1}) holds. By Proposition \ref{prop 4.2} there exist balanced $\Z_2$-graded triples $(L_{V_E}, g^{L_{V_E}}, \nabla^{L_{V_E}})$ and $(L_{V_F}, g^{L_{V_F}}, \nabla^{L_{V_F}})$ such that $L_{V_E}\to B$ and $L_{V_F}\to B$ satisfy the MF property for $\DD^{S\otimes V_E}$ and $\DD^{S\otimes V_F}$, respectively, and
\begin{equation}\label{eq 4.2.6}
\wh{\eta}^{V_E}(g^{V_E}, \nabla^{V_E}, L_{V_E})\equiv 0\equiv\wh{\eta}^{V_F}(g^{V_F}, \nabla^{V_F}, L_{V_F}).
\end{equation}

Let $(L_{E^+}, g^{L_{E^+}}, \nabla^{L_{E^+}})$ and $(L_{E^-}, g^{L_{E^-}}, \nabla^{L_{E^-}})$ be $\Z_2$-graded triples so that $L_{E^+}\to B$ and $L_{E^-}\to B$ satisfy the MF property for $\DD^{S\otimes E^+}$ and $\DD^{S\otimes E^-}$, respectively. By Lemma \ref{lemma 4.1}, $(L_{E^+}\oplus L_{E^-}^{\op}, g^{L_{E^+}}\oplus g^{L_{E^-}, \op}, \nabla^{L_{E^+}}\oplus\nabla^{L_{E^-}, \op})$ is a $\Z_2$-graded triple so that $L_{E^+}\oplus L_{E^-}^{\op}\to B$ satisfies the MF property for $\DD^{S\otimes E}$. By Proposition \ref{prop 4.1},
$$\big((L_{E^+}\oplus L_{E^-}^{\op})\oplus L_{V_E}, (g^{L_{E^+}}\oplus g^{L_{E^-}, \op})\oplus g^{L_{V_E}}, (\nabla^{L_{E^+}}\oplus\nabla^{L_{E^-}, \op})\oplus\nabla^{L_{V_E}}\big)$$
is a $\Z_2$-graded triple so that $(L_{E^+}\oplus L_{E^-}^{\op})\oplus L_{V_E}\to B$ satisfies the MF property for $\DD^{S\otimes(E\oplus V_E)}$.

Let $(L_{F^+}, g^{L_{F^+}}, \nabla^{L_{F^+}})$ and $(L_{F^-}, g^{L_{F^-}}, \nabla^{L_{F^-}})$ be $\Z_2$-graded triples so that $L_{F^+}\to B$ and $L_{F^-}\to B$ satisfy the MF property for $\DD^{S\otimes F^+}$ and $\DD^{S\otimes F^-}$, respectively. Similarly,
$$\big((L_{F^+}\oplus L_{F^-}^{\op})\oplus L_{V_F}, (g^{L_{F^+}}\oplus g^{L_{F^-}, \op})\oplus g^{L_{V_F}}, (\nabla^{L_{F^+}}\oplus\nabla^{L_{F^-}, \op})\oplus\nabla^{L_{V_F}}\big)$$
is a $\Z_2$-graded triple so that $(L_{F^+}\oplus L_{F^-}^{\op})\oplus L_{V_F}\to B$ satisfies the MF property for $\DD^{S\otimes(F\oplus V_F)}$.

By applying Theorem \ref{thm 3.1} to the $\Z_2$-graded isometric isomorphism $\alpha$, there exist balanced $\Z_2$-graded triples $(W_E, g^{W_E}, \nabla^{W_E})$ and $(W_F, g^{W_F}, \nabla^{W_F})$ and a $\Z_2$-graded isometric isomorphism
\begin{equation}\label{eq 4.2.7}
\begin{tikzcd}
\big((L_{E^+}\oplus L_{E^-}^{\op})\oplus L_{V_E}\oplus W_E, (g^{L_{E^+}}\oplus g^{L_{E^-}, \op})\oplus g^{L_{V_E}}\oplus g^{W_E})\big) \arrow{d}{h} \\ \big((L_{F^+}\oplus L_{F^-}^{\op})\oplus L_{V_F}\oplus W_F, (g^{L_{F^+}}\oplus g^{L_{F^-}, \op})\oplus g^{L_{V_F}}\oplus g^{V_F})\big)
\end{tikzcd}
\end{equation}
such that
\begin{equation}\label{eq 4.2.8}
\begin{split}
&\wh{\eta}^{F\oplus V_F}(g^F\oplus g^{V_F}, \nabla^F\oplus\nabla^{V_F}, (L_{F^+}\oplus L_{F^-}^{\op})\oplus L_{V_F})\\
&\quad-\wh{\eta}^{E\oplus V_E}(g^E\oplus g^{V_E}, \nabla^E\oplus\nabla^{V_E}, (L_{E^+}\oplus L_{E^-}^{\op})\oplus L_{V_E})\\
&\equiv\int_{X/B}\todd(\nabla^{T^VX})\wedge\CS(\nabla^E\oplus\nabla^{V_E}, \alpha^*(\nabla^F\oplus\nabla^{V_F}))\\
&\quad-\CS\big((\nabla^{L_{E^+}}\oplus\nabla^{L_{E^-}, \op})\oplus\nabla^{L_{V_E}}\oplus\nabla^{W_E}, h^*((\nabla^{L_{F^+}}\oplus\nabla^{L_{F^-}, \op})\oplus\nabla^{L_{V_F}}\oplus\nabla^{W_F})\big).
\end{split}
\end{equation}
By Proposition \ref{prop 4.1} and (\ref{eq 4.2.6}),
\begin{equation}\label{eq 4.2.9}
\wh{\eta}^{E\oplus V_E}(g^E\oplus g^{V_E}, \nabla^E\oplus\nabla^{V_E}, (L_{E^+}\oplus L_{E^-}^{\op})\oplus L_{V_E})\equiv\wh{\eta}^E(g^E, \nabla^E, L_{E^+}\oplus L_{E^-}^{\op}),
\end{equation}
and similarly,
\begin{equation}\label{eq 4.2.10}
\wh{\eta}^{F\oplus V_F}(g^F\oplus g^{V_F}, \nabla^F\oplus\nabla^{V_F}, (L_{F^+}\oplus L_{F^-}^{\op})\oplus L_{V_F})\equiv\wh{\eta}^F(g^F, \nabla^F, L_{F^+}\oplus L_{F^-}^{\op}).
\end{equation}
By (\ref{eq 4.2.1}), (\ref{eq 4.2.9}) and (\ref{eq 4.2.10}), (\ref{eq 4.2.8}) becomes
\begin{equation}\label{eq 4.2.11}
\begin{split}
&\wh{\eta}^F(g^F, \nabla^F, L_{F^+}\oplus L_{F^-}^{\op})-\wh{\eta}^E(g^E, \nabla^E, L_{E^+}\oplus L_{E^-}^{\op})\equiv\int_{X/B}\todd(\nabla^{T^VX})\wedge(\omega_E-\omega_F)\\
&\quad-\CS\big((\nabla^{L_{E^+}}\oplus\nabla^{L_{E^-}, \op})\oplus\nabla^{L_{V_E}}\oplus\nabla^{W_E}, h^*((\nabla^{L_{F^+}}\oplus\nabla^{L_{F^-}, \op})\oplus\nabla^{L_{V_F}}\oplus\nabla^{W_F})\big).
\end{split}
\end{equation}

Since the $\Z_2$-graded triples $(L_{V_E}\oplus W_E, g^{L_{V_E}}\oplus g^{W_E}, \nabla^{L_{V_E}}\oplus\nabla^{W_E})$ and $(L_{V_F}\oplus W_F, g^{L_{V_F}}\oplus g^{W_F}, \nabla^{L_{V_F}}\oplus\nabla^{W_F})$ are balanced, it follows from (\ref{eq 4.2.7}) and (\ref{eq 4.2.11}) that (\ref{eq 4.2.5}) holds.
\end{proof}

\subsection{The RRG theorem in $\R/\Z$ $K$-theory for the twisted spin$^c$ Dirac operators}\label{s 4.3}

In this subsection we give an alternative proof of the RRG theorem in $\R/\Z$ $K$-theory for twisted spin$^c$ Dirac operators without the kernel bundle assumption (Theorem \ref{thm 1.3}).

Define a map $\ch_{\wh{K}}:\wh{K}_{\FL}(X)\to\Omega^{\even}_\Q(X)$ by $\ch_{\wh{K}}(\E)=\ch(\nabla^E)+d\omega$. The $\R/\Z$ $K$-group $K^{-1}_{\LL}(X)$ \cite[Definition 7]{L94} (cf. \cite[(2.20)]{FL10}) can be defined as $K^{-1}_{\LL}(X)=\ker(\ch_{\wh{K}})$. Thus a $\Z_2$-graded generator $\E$ of $K^{-1}_{\LL}(X)$ is a $\Z_2$-graded generator of $\wh{K}_{\FL}(X)$ satisfying
\begin{equation}\label{eq 4.3.1}
\ch(\nabla^{E, +})-\ch(\nabla^{E, -})=-d\omega.
\end{equation}
Note that (\ref{eq 4.3.1}) implies that $\rk(E^+)=\rk(E^-)$.

The $\R/\Q$ Chern character $\ch_{\R/\Q}:K^{-1}_{\LL}(X)\to H^{\odd}(X; \R/\Q)$ is defined as follows. Let $\E$ be a $\Z_2$-graded generator of $K^{-1}_{\LL}(X)$. By (\ref{eq 4.3.1}), there exist a $k\in\N$ and an isometric isomorphism $\alpha:(kE^+, kg^{E, +})\to(kE^-, kg^{E, -})$ (see \cite[Remark 1]{H20} for a proof). Define $\ch_{\R/\Q}(\E)$ by
\begin{equation}\label{eq 4.3.2}
\ch_{\R/\Q}(\E)=\bigg[\frac{1}{k}\CS(\alpha^*(k\nabla^{E, -}), k\nabla^{E, +})+\omega\bigg]\mod\Q.
\end{equation}
It is easy to check that the odd form of the right-hand side of (\ref{eq 4.3.2}) is closed. Note that $\ch_{\R/\Q}(\E)$ is independent of the choices of $k$ and $\alpha$ \cite[p.289]{L94}.

The analytic index $\ind^a_{\R/\Z}$ in $\R/\Z$ $K$-theory of a $\Z_2$-graded generator $\E$ of $K^{-1}_{\LL}(X)$ is defined by the same formula (\ref{eq 4.2.2}). It is easy to check that $\ind^a_{\R/\Z}(\E)\in K^{-1}_{\LL}(B)$. As an immediate consequence of Proposition \ref{prop 1.2}, the analytic index in $\R/\Z$ $K$-theory $$\ind^a_{\R/\Z}:K^{-1}_{\LL}(X)\to K^{-1}_{\LL}(B)$$
is a well-defined group homomorphism.

We now prove Theorem \ref{thm 1.3}.
\begin{proof}[Proof of Theorem \ref{thm 1.3}]
As mentioned in \S1, we prove (\ref{eq 1.0.3}) at the differential form level. Let $\E$ be a $\Z_2$-graded generator of $K^{-1}_{\LL}(X)$. By (\ref{eq 4.3.1}), there exists a $k_1\in\N$ such that $k_1E^+\cong k_1E^-$.

Let $(L_{E^+}, g^{L_{E^+}}, \nabla^{L_{E^+}})$ and $(L_{E^-}, g^{L_{E^-}}, \nabla^{L_{E^-}})$ be $\Z_2$-graded triples so that $L_{E^+}\to B$ and $L_{E^-}\to B$ satisfy the MF property for $\DD^{S\otimes E^+}$ and $\DD^{S\otimes E^-}$, respectively. By Lemma \ref{lemma 4.1},
$$(L_{E^+}\oplus L_{E^-}^{\op}, g^{L_{E^+}}\oplus g^{L_{E^-}, \op}, \nabla^{L_{E^+}}\oplus\nabla^{L_{E^-}, \op})$$
is a $\Z_2$-graded triple so that $L_{E^+}\oplus L_{E^-}^{\op}\to B$ satisfies the MF property for $\DD^{S\otimes E}$. By \cite[Corollary 1]{H20}, $\ind^a_{\R/\Z}(\E)$ is given by
\begin{equation}\label{eq 4.3.3}
\begin{split}
\ind^a_{\R/\Z}(\E)&=\bigg(L_{E^+}\oplus L_{E^-}^{\op}, g^{L_{E^+}}\oplus g^{L_{E^-}, \op}, \nabla^{L_{E^+}}\oplus\nabla^{L_{E^-}, \op},\\ &\qquad\int_{X/B}\todd(\nabla^{T^VX})\wedge\omega+\wh{\eta}^E(g^E, \nabla^E, L_{E^+}\oplus L_{E^-}^{\op})\bigg).
\end{split}
\end{equation}
On the other hand, by (\ref{eq 3.1.17}), (\ref{eq 4.3.1}) and Theorem \ref{thm 4.1},
\begin{displaymath}
\begin{split}
\ch(\nabla^{L_{E^+}})-\ch(\nabla^{L_{E^-}})&=\int_{X/B}\todd(\nabla^{T^VX})\wedge\big(\ch(\nabla^{E, +})-\ch(\nabla^{E, -})\big)\\
&\qquad-d\wh{\eta}^{E^+}(g^{E, +}, \nabla^{E, +}, L_{E^+})+d\wh{\eta}^{E^-}(g^{E, -}, \nabla^{E, -}, L_{E^-})\\
&=\int_{X/B}\todd(\nabla^{T^VX})\wedge(-d\omega)-d\wh{\eta}^E(g^E, \nabla^E, L_{E^+}\oplus L_{E^-}^{\op})\\
&=-d\bigg(\int_{X/B}\todd(\nabla^{T^VX})\wedge\omega+\wh{\eta}^E(g^E, \nabla^E, L_{E^+}\oplus L_{E^-}^{\op})\bigg).
\end{split}
\end{displaymath}
Thus there exists a $k_2\in\N$ such that $k_2L_{E^+}\cong k_2L_{E^-}$.

Let $k$ be the least common multiple of $k_1$ and $k_2$. Let $\alpha:(kE^+, kg^{E, +})\to(kE^-, kg^{E, -})$ be an isometric isomorphism. By Theorem \ref{thm 3.1}, there exist balanced $\Z_2$-graded triples $(W_+, g^{W_+}, \nabla^{W_+})$ and $(W_-, g^{W_-}, \nabla^{W_-})$ and a $\Z_2$-graded isometric isomorphism $h:(L_{kE^+}\oplus W_+, g^{L_{kE^+}}\oplus g^{W_+})\to(L_{kE^-}\oplus W_-, g^{L_{kE^-}}\oplus g^{W_-})$ such that
\begin{equation}\label{eq 4.3.4}
\begin{split}
&\wh{\eta}^{kE^-}(kg^{E, -}, k\nabla^{E, -}, L_{kE^-})-\wh{\eta}^{kE^+}(kg^{E, +}, k\nabla^{E, +}, L_{kE^+})\\
&\equiv\int_{X/B}\todd(\nabla^{T^VX})\wedge\CS(k\nabla^{E, +}, \alpha^*(k\nabla^{E, -}))-\CS(\nabla^{L_{kE^+}}\oplus\nabla^{W_+}, h^*(\nabla^{L_{kE^-}}\oplus\nabla^{W_-})).
\end{split}
\end{equation}
Recall from (\ref{eq 4.1.3}) that
\begin{equation}\label{eq 4.3.5}
L_{kE^\pm}=kL_{E^\pm},\qquad g^{L_{kE^\pm}}=kg^{L_{E^\pm}},\qquad\nabla^{L_{kE^\pm}}=k\nabla^{L_{E^\pm}}.
\end{equation}
By applying Remark \ref{remark 3.1} to $kL_{E^+}\cong kL_{E^-}$, we take $W_+\to B$ and $W_-\to B$ to be the zero bundle. By (\ref{eq 4.3.5}), (\ref{eq 4.3.4}) becomes
\begin{equation}\label{eq 4.3.6}
\begin{split}
&\wh{\eta}^{kE^-}(kg^{E, -}, k\nabla^{E, -}, kL_{E^-})-\wh{\eta}^{kE^+}(kg^{E, +}, k\nabla^{E, +}, kL_{E^+})\\
&\equiv\int_{X/B}\todd(\nabla^{T^VX})\wedge\CS(k\nabla^{E, +}, \alpha^*(k\nabla^{E, -}))-\CS(k\nabla^{L_{E^+}}, h^*(k\nabla^{L_{E^-}})).
\end{split}
\end{equation}
By Proposition \ref{prop 4.1} and Theorem \ref{thm 4.1},
\begin{equation}\label{eq 4.3.7}
\begin{split}
&\wh{\eta}^{kE^-}(kg^{E, -}, k\nabla^{E, -}, kL_{E^-})-\wh{\eta}^{kE^+}(kg^{E, +}, k\nabla^{E, +}, kL_{E^+})\\
&\equiv k\big(\wh{\eta}^{E^-}(g^{E, -}, \nabla^{E, -}, L_{E^-})-\wh{\eta}^{E^+}(g^{E, +}, \nabla^{E, +}, L_{E^+})\big)\\
&\equiv-k\wh{\eta}^E(g^E, \nabla^E, L_{E^+}\oplus L_{E^-}^{\op}).
\end{split}
\end{equation}

On the other hand, denote by $h_+:(kL_{E^+}^+, kg^{L_{E^+}, +})\to(kL_{E^-}^+, kg^{L_{E^-}, +})$ and $h_-:(kL_{E^+}^-, kg^{L_{E^+}, -})\to(kL_{E^-}^-, kg^{L_{E^-}, -})$ the even and the odd part of $h$, respectively. Then
$$h_+\oplus h_-^{-1}:(kL_{E^+}^+\oplus kL_{E^-}^-, kg^{L_{E^+}, +}\oplus kg^{L_{E^-}, -})\to(kL_{E^-}^+\oplus kL_{E^+}^-, kg^{L_{E^-}, +}\oplus kg^{L_{E^+}, -})$$
is an isometric isomorphism. By (\ref{eq 2.2.7}), (\ref{eq 2.2.4}) and (\ref{eq 2.2.6}),
\begin{align}
&\CS\big(k\nabla^{L_{E^+}}, h^*(k\nabla^{L_{E^-}})\big)\nonumber\\
&=\CS\big(k\nabla^{L_{E^+}, +}\oplus k\nabla^{L_{E^+}, -}, h_+^*(k\nabla^{L_{E^-}, +})\oplus h_-^*(k\nabla^{L_{E^-}, -})\big)\nonumber\\
&\equiv\CS\big(k\nabla^{L_{E^+}, +}, h_+^*(k\nabla^{L_{E^-}, +})\big)-\CS\big(k\nabla^{L_{E^+}, -}, h_-^*(k\nabla^{L_{E^-}, -})\big)\nonumber\\
&\equiv\CS\big(k\nabla^{L_{E^+}, +}, h_+^*(k\nabla^{L_{E^-}, +})\big)-\CS\big((h_-^{-1})^*(k\nabla^{L_{E^+}, -}), k\nabla^{L_{E^-}, -}\big)\nonumber\\
&\equiv\CS\big(k\nabla^{L_{E^+}, +}, h_+^*(k\nabla^{L_{E^-}, +})\big)+\CS\big(k\nabla^{L_{E^-}, -}, (h_-^{-1})^*(k\nabla^{L_{E^+}, -})\big)\nonumber\\
&\equiv\CS\big(k\nabla^{L_{E^+}, +}\oplus k\nabla^{L_{E^-}, -}, h_+^*(k\nabla^{L_{E^-}, +})\oplus(h_-^{-1})^*(k\nabla^{L_{E^+}, -})\big)\nonumber\\
&\equiv\CS\big(k(\nabla^{L_{E^+}, +}\oplus\nabla^{L_{E^-}, -}), (h_+\oplus h_-^{-1})^*(k(\nabla^{L_{E^-}, +}\oplus\nabla^{L_{E^+}, -}))\big).\label{eq 4.3.8}
\end{align}
By (\ref{eq 4.3.7}) and (\ref{eq 4.3.8}), (\ref{eq 4.3.6}) becomes
\begin{displaymath}
\begin{split}
&-k\wh{\eta}^E(g^E, \nabla^E, L_{E^+}\oplus L_{E^-}^{\op})+\CS\big(k(\nabla^{L_{E^+}, +}\oplus\nabla^{L_{E^-}, -}), (h_+\oplus h_-^{-1})^*(k(\nabla^{L_{E^-}, +}\oplus\nabla^{L_{E^+}, -}))\big)\\
&\equiv\int_{X/B}\todd(\nabla^{T^VX})\wedge\CS(k\nabla^{E, +}, \alpha^*(k\nabla^{E, -})).
\end{split}
\end{displaymath}
By (\ref{eq 2.2.4}) it becomes
\begin{equation}\label{eq 4.3.9}
\begin{split}
&k\wh{\eta}^E(g^E, \nabla^E, L_{E^+}\oplus L_{E^-}^{\op})+\CS\big((h_+\oplus h_-^{-1})^*(k(\nabla^{L_{E^-}, +}\oplus\nabla^{L_{E^+}, -})), k(\nabla^{L_{E^+}, +}\oplus\nabla^{L_{E^-}, -})\big)\\
&\equiv\int_{X/B}\todd(\nabla^{T^VX})\wedge\CS(\alpha^*(k\nabla^{E, -}), k\nabla^{E, +}).
\end{split}
\end{equation}
By first dividing both sides of (\ref{eq 4.3.9}) by $k$, and adding \dis{\int_{X/B}\todd(\nabla^{T^VX})\wedge\omega} to both sides of (\ref{eq 4.3.9}), (\ref{eq 4.3.9}) becomes
\begin{equation}\label{eq 4.3.10}
\begin{split}
&\frac{1}{k}\CS\big((h_+\oplus h_-^{-1})^*(k(\nabla^{L_{E^-}, +}\oplus\nabla^{L_{E^+}, -})), k(\nabla^{L_{E^+}, +}\oplus\nabla^{L_{E^-}, -})\big)\\
&\qquad+\int_{X/B}\todd(\nabla^{T^VX})\wedge\omega+\wh{\eta}^E(g^E, \nabla^E, L_{E^+}\oplus L_{E^-}^{\op})\\
&\equiv\int_{X/B}\todd(\nabla^{T^VX})\wedge\bigg(\frac{1}{k}\CS(\alpha^*(k\nabla^{E, -}), k\nabla^{E, +})+\omega\bigg).
\end{split}
\end{equation}
Since the left-hand side of (\ref{eq 4.3.10}) is a differential form representative of $\ch_{\R/\Q}(\ind^a_{\R/\Z}(\E))$ and the right-hand side of (\ref{eq 4.3.10}) is that of \dis{\int_{X/B}\todd(T^VX)\cup\ch_{\R/\Q}(\E)}, (\ref{eq 1.0.3}) holds.
\end{proof}

Note that (\ref{eq 4.3.10}) is a refinement of (\ref{eq 1.0.3}) at the differential form level.

\bibliographystyle{amsplain}
\bibliography{MBib}
\end{document}